\documentclass[10pt,reqno]{amsart}

\usepackage{amssymb, amsmath, amsfonts, latexsym, mathtools}
\usepackage{graphicx}
\usepackage{xcolor}
\usepackage{amscd}

\usepackage[top=30truemm,bottom=30truemm,left=25truemm,right=25truemm]{geometry}

\usepackage{hyperref}
\usepackage[nameinlink]{cleveref}
\hypersetup{
setpagesize=false,
bookmarksnumbered=true,%
bookmarksopen=true,%
colorlinks=true,%
linkcolor=orange,
citecolor=blue,
}

\crefname{equation}{}{}

\newtheorem{thm}{Theorem}[section]
\crefname{thm}{Theorem}{Theorems}
\newtheorem{prop}[thm]{Proposition}
\crefname{prop}{Proposition}{Propositions}
\newtheorem{lem}[thm]{Lemma}
\crefname{lem}{Lemma}{Lemmas}
\newtheorem{cor}[thm]{Corollary}
\crefname{cor}{Corollary}{Corollarys}
\newtheorem{conj}[thm]{Conjecture}

\theoremstyle{definition}
\newtheorem*{definition}{Definition}
\newtheorem*{example}{Example}

\theoremstyle{remark}
\newtheorem*{remark}{Remark}

\newcommand{\til}{\widetilde}
\newcommand{\wh}{\widehat}
\newcommand{\vphi}{\varphi}
\newcommand{\ep}{\varepsilon}
\newcommand{\rank}{\operatorname{rank}}
\newcommand{\ord}{\operatorname{ord}}
\newcommand{\Img}{\operatorname{Im}}
\newcommand{\Ker}{\operatorname{Ker}}
\newcommand{\modd}{\operatorname{ mod }}
\newcommand{\midd}{\,\middle|\,}
\newcommand{\ceq}{\coloneqq}
\newcommand{\DHk}{\mathcal{D}\til{\mathcal{H}}_k}
\newcommand{\DTk}{\mathcal{D}\til{\mathcal{T}}_k}
\newcommand{\Real}{\operatorname{Re}}

\usepackage[OT2,T1]{fontenc}
\DeclareSymbolFont{cyrletters}{OT2}{wncyr}{m}{n}
\DeclareMathSymbol{\sh}{\mathalpha}{cyrletters}{"78}

\numberwithin{equation}{section}

\begin{document}

\title{DOUBLE EISENSTEIN SERIES AND MODULAR FORMS OF LEVEL 4}
\author{Katsumi Kina}
\address{Graduate School of Mathematics, Kyushu University, Motooka 744, Nishi-ku, Fukuoka
819-0395, Japan}
\email{kkina@math.kyushu-u.ac.jp}

\maketitle

\begin{abstract}
We study the $\mathbb{Q}$-vector space generated by the double zeta values with character of conductor 4. For this purpose, we define associated double Eisenstein series and investigate their relation with modular forms of level 4.
\\
\\
\end{abstract}

\hypersetup{linkcolor=black}
\tableofcontents
\hypersetup{linkcolor=orange}

\section{Introduction and main results}

In \cite{GKZ}, Gangl, Kaneko, and Zagier studied the ``double shuffle relations'' satisfied by the double zeta values 
\begin{equation*}
\zeta(r,s) \ceq \sum_{0<m<n}\frac{1}{m^{r}n^{s}}\quad (r\geq 1,s\geq 2),
\end{equation*}
and revealed the relationship between the space of double zeta values and the period polynomials of modular forms. These results give a conjecturally sharp upper bound on the dimension of the $\mathbb{Q}$-vector space generated by double zeta values. They also defined the ``double Eisenstein series'' and confirmed that they satisfy double shuffle relations.

In \cite{K-T}, Kaneko and Tasaka considered the double zeta values of level $2$
\begin{equation*}
\zeta^{e,f}(r,s) \ceq \sum_{\substack{0<m<n\\m\equiv e,\ n\equiv f\modd 2}}\frac{1}{m^{r}n^{s}}\quad (e,f\in\{0,1\},\, r\geq 1,\, s\geq 2),
\end{equation*}
and studied the formal double zeta space. Furthermore, they defined the ``double Eisenstein series of level $2$'' and showed that they also satisfy the double shuffle relations and obtained the relationship between double zeta values of level $2$ and modular forms of level $2$, as in the case of \cite{GKZ}.

In \cite{K-Tsu}, Kaneko and Tsumura introduced the multiple $\til{T}$-values\footnote{Their definition is without the factor $(2\pi i)^{-(k_1 + \ldots + k_r)}$.} defined by
\begin{equation*}
\til{T}(k_1,\dots,k_r) \ceq \frac{2^{r}}{(2\pi i)^{k_1+\cdots +k_r}}\sum_{0<n_1<\cdots<n_r}\frac{\chi_4(n_1)\chi_4(n_2-n_1)\cdots\chi_4(n_r-n_{r-1})}{n_1^{k_1}n_2^{k_2}\cdots n_r^{k_r}},
\end{equation*}
$\chi_4$ being the non-trivial Dirichlet character of conductor $4$, and showed interesting relations they satisfy. Moreover, they made a conjecture on the dimension of the $\mathbb{Q}$-vector space generated by double $\til{T}$-values weight $k$:
$$\DTk \ceq \left\langle\til{T}(r,k-r)\midd 1\leq r\leq k-1\right\rangle_{\mathbb{Q}},$$
 and on a relation between period polynomials of modular forms of level $4$ and double $\til{T}$-values. A precise statement of this conjecture is presented in \Cref{sec:relation}.

In the present paper, we give a partial solution to their conjecture. For this purpose, we define a function $\til{H}_{k_1,\dots,k_r}$ as a variant of the double Eisenstein series and show that it has connections with modular forms of level $4$.

\begin{definition}
For any integers $k_1,\dots,k_{r-1}\geq 2$ and $k_r\geq 3$, we define the function $\til{H}_{k_1,\dots,k_r}$ on the upper half-plane $\mathbb{H}$ by
\begin{equation*}
\til{H}_{k_1,\dots,k_r}(\tau)
\ceq \frac{2^r}{(2\pi i)^{k_1+\cdots+k_r}}\sum_{0\prec4m_1\tau+n_1\prec\cdots\prec4m_r\tau+n_r}\frac{\chi_4(n_1)\chi_4(n_2-n_1)\cdots\chi_4(n_r-n_{r-1})}{(4m_1\tau+n_1)^{k_1}\cdots(4m_r\tau+n_r)^{k_r}}.
\end{equation*}
Here, we define an order $\prec$ on points in $\mathbb{H}$ as 
\begin{align*}0\prec m\tau+n
&\overset{\text{def}}{\Longleftrightarrow}
\left\{\begin{array}{cc}0<m\\ \text{or}\\ m=0\text{ and }0<n\end{array}\right. ,
\\
m_1\tau+n_1\prec m_2\tau+n_2
&\overset{\text{def}}{\Longleftrightarrow}
0\prec(m_2-m_1)\tau+(n_2-n_1).
\end{align*}
\end{definition}

This series converges absolutely and locally uniformly, thus $\til{H}_{k_1,\cdots,k_r}(\tau)$ is a holomorphic function on $\mathbb{H}$. In the following, we consider the cases $r=1$ and $r=2$. In \Cref{sec:expansion}, we extend the definitions of $\til{H}_k$ and $\til{H}_{k_1,k_2}$ to all $k,k_1,k_2\geq 1$ using the Fourier expansion. Additionally, we show that the constant terms of $\til{H}_{k}(\tau)$ and $\til{H}_{k_1,k_2}(\tau)$ are $\til{T}(k)$ and $\til{T}(k_1,k_2)$ respectively.

\begin{thm}\label{thm:shuffle}
For any integers $k_1,k_2\geq 1$, we have the shuffle relations
\begin{equation}\label{eq:shuffle}
\til{H}_{k_1}(\tau)\til{H}_{k_2}(\tau)=\sum_{p=1}^{k_1+k_2-1}\left(\binom{p-1}{k_1-1}+\binom{p-1}{k_2-1}\right)\til{H}_{k_1+k_2-p,p}(\tau).
\end{equation}
Furthermore, for any even integer $k\geq 4$, we have
\begin{align}
\til{G}_{k}(\tau)
&=\frac{1}{k-1}\left(\sum_{p=1}^{k-1}2^{k-2-p}\til{H}_{p,k-p}(\tau)+\frac{1}{2}\til{H}_{k-1,1}(\tau)\right)\label{eq:eisen1}
\intertext{and}
\til{G}_{k}(\tau)
&=\frac{1}{2(k-1)}\sum_{p=1:\text{odd}}^{k-1}\til{H}_{k-p}(\tau)\til{H}_{p}(\tau),\label{eq:eisen2}
\end{align}
where
\begin{equation}
\til{G}_{k}(\tau) \ceq  \frac{1}{(2\pi i)^{k}}\sum_{\substack{0<4m\tau+n\\n\equiv 1,3\modd 4}}\frac{1}{(4m\tau+n)^{k}}.
\end{equation}
\end{thm}

If $k$ and $N$ are positive integers and $\chi$ is a Dirichlet character modulo $N$, then we denote by $M_k(\Gamma_0(N),\chi)$ the $\mathbb{C}$-vector space of holomorphic modular forms of weight $k$ and level $N$ with character $\chi$, and by $S_k(\Gamma_0(N),\chi)$ the subspace of cusp forms. Furthermore, $M_k^{\mathbb{Q}}(\Gamma_0(4))$ is a subspace of $M_k(\Gamma_0(4))$ consisting of forms with rational Fourier coefficients. The subspace $M_k^{\mathbb{Q}}(\Gamma_0(N))\cap S_k(\Gamma_0(N))$ is denoted by $S_k^{\mathbb{Q}}(\Gamma_0(N))$.

We can check easily $\til{H}_k(\tau)\in M_k(\Gamma_0(4),\chi_4)$ for any odd integer $k\geq 3$, thus $\til{H}_r(\tau)\til{H}_{k-r}(\tau)\in M_k(\Gamma_0(4))$ for any even integer $k$ and any odd integer $r$ with $3\leq r\leq k-3$. Furthermore, we have $\til{G}_k(\tau)\in M_k(\Gamma_0(4))$ for any even integer $k\geq 4$. For an integer $k\geq 2$, let $\DHk$ denote the $\mathbb{Q}$-vector space generated by $\til{H}_{r,k-r}(\tau)$:
$$\DHk \ceq \left\langle\til{H}_{r,k-r}(\tau)\midd 1\leq r\leq k-1\right\rangle_{\mathbb{Q}}.$$
Moreover define $M\DHk \ceq \DHk\cap M_k^{\mathbb{Q}}(\Gamma_0(4))$ and $S\DHk \ceq \DHk\cap S_k^{\mathbb{Q}}(\Gamma_0(4))$.

\begin{thm}\label{thm:dimension}
For any integer $k\geq 2$,
\begin{equation}\label{eq:all-dimension}
\dim_{\mathbb{Q}}\DHk=k-1.
\end{equation}
Furthermore, for any even integer $k\geq 4$, we have
\begin{align}
\dim_{\mathbb{Q}}S\DHk
&= \left[\frac{k-2}{4}\right]=\dim S_k(\Gamma_0(4))-\dim S_k(\Gamma_0(2)),
\label{eq:even-dimension}
\\
\dim_{\mathbb{Q}}M\DHk
&=\dim_{\mathbb{Q}}S\DHk +1.\label{eq:mod-cusp-dim}
\end{align}
Specifically, $M\DHk=\mathbb{Q}\cdot\til{G}_k(\tau)\oplus S\DHk$ and a basis of $M\DHk$ is given by
\begin{equation}\label{H-base}
B_{M\DHk} \ceq \left\{\til{G}_k(\tau),\til{H}_{r}(\tau)\til{H}_{k-r}(\tau) \midd 3\leq r\leq k-3:\text{odd}\right\}.
\end{equation}
\end{thm}

\begin{cor}\label{cor:Ttil-dim}
For any enen integer $k\geq 2$,
\begin{equation}\label{eq:Ttil-dim}
\dim_{\mathbb{Q}}\DTk\leq k-1-\left[\frac{k-2}{4}\right].
\end{equation}
\end{cor}

Kaneko and Tsumura conjecture the equality holds in \cref{eq:Ttil-dim}.

We prove \cref{thm:shuffle} in \Cref{sec:proof1}. Our proof is accomplished by performing intricate calculations using the Fourier expansions of $\til{H}_{k}$ and $\til{H}_{k_1,k_2}$. In \Cref{sec:ImDH}, we calculate the dimension of the ``space of imaginary part of $\DHk$''. We prove \cref{thm:dimension} in \Cref{sec:proof2} by proving that the elements of $B_{M\DHk}$ are $\mathbb{Q}$-linearly independent. For this purpose, we calculate periods of certain modular forms.

\section*{Acknowledgement}
\cref{thm:shuffle,thm:dimension} were conjectured by Professor Koji Tasaka. He not only provided the conjecture but also engaged in discussions with the author, offering various pieces of advice. The author would like to express his sincere gratitude to Prof. Tasaka.

He would also like to thank Professor Masanobu Kaneko for his guidance and numerous corrections made to this paper. Furthermore, Kaneko's unpublished notes in the original case (of level $1$) are immensely helpful in proving \cref{thm:shuffle}. Particularly, the proof outlined in those notes are directly applicable to the proof of the shuffle equation \cref{eq:shuffle}.

He would also like to express his sincere gratitude to Professor Henrik Bachmann and the members of the Kaneko Laboratory for advising him on various aspects of his research.

He wishes to express his sincere gratitude to the referee for the important suggestions and valuable comments.

\section{The Fourier expansions of $\til{H}_{k}(\tau)$ and $\til{H}_{k_1,k_2}(\tau)$}\label{sec:expansion}

Let $\chi_0$ be the trivial character modulo $4$. For an integer $k\geq 1$, define
\begin{align*}
\vphi_k(\tau) &\ceq \frac{(-2\pi i)^k}{(k-1)!}\left(\frac{1}{2}\delta_{k,1}+\frac{1}{2^{k-1}}\sum_{n=1}^{\infty}\chi_0(n+1)n^{k-1}q^{n}\right),
\\
\psi_k(\tau) &\ceq \frac{(-2\pi i)^k}{(k-1)!}\frac{1}{2^{k-1}}\sum_{n=1}^{\infty}\chi_{0}(n)n^{k-1}q^{n},
\\
\til{\vphi}_k(\tau) &\ceq \frac{(-2\pi i)^k}{(k-1)!}\left(\frac{1}{2}\delta_{k,1}+\frac{1}{2^{k-1}}\sum_{n=1}^{\infty}\chi_4(n+1)n^{k-1}q^{n}\right),
\\
\til{\psi}_k(\tau) &\ceq \frac{(-2\pi i)^k}{(k-1)!}\frac{i}{2^{k-1}}\sum_{n=1}^{\infty}\chi_{4}(n)n^{k-1}q^{n},\quad\quad (q=e^{2\pi i\tau} ,\tau\in\mathbb{H}).
\end{align*}
By using the standard partial fractional decomposition formula of trigonometric functions, we can find that
\begin{equation}\label{eq:vphi-psi}
\begin{array}{ccccc}
&\displaystyle{\vphi_k(\tau)=2^k\sum_{n\in\mathbb{Z}}\frac{\chi_0(n+1)}{(4\tau+n)^k}}&, &
\displaystyle{\psi_k(\tau)=2^k\sum_{n\in\mathbb{Z}}\frac{\chi_4(n+1)}{(4\tau+n)^k}}
\\
&\displaystyle{\til{\vphi}_k(\tau)=2^k\sum_{n\in\mathbb{Z}}\frac{\chi_0(n)}{(4\tau+n)^k}}&, & 
\displaystyle{\til{\psi}_k(\tau)=2^k\sum_{n\in\mathbb{Z}}\frac{\chi_4(n)}{(4\tau+n)^k}}
\end{array}
\quad\left(\text{if }k=1,\ \sum_{n\in\mathbb{Z}}=\lim_{N\to\infty}\sum_{n=-N}^N\right).
\end{equation}
Moreover define
\begin{align*}
g_k(\tau)&\ceq\sum_{m=1}^{\infty}\vphi_k(m\tau)
\ ,\quad \til{g}_k(\tau)\ceq\sum_{m=1}^{\infty}\til{\vphi}_k(m\tau)\quad\quad (k\geq 2),
\\
h_k(\tau)&\ceq\sum_{m=1}^{\infty}\psi_k(m\tau)
\ ,\quad \til{h}_k(\tau)\ceq\sum_{m=1}^{\infty}\til{\psi}_k(m\tau)\quad\quad (k\geq 1).
\end{align*}

\begin{prop}\label{prop:fourier-exp}
For any integer $k\geq 3$, we have
\begin{equation}\label{eq:fourier-exp1}
\til{H}_k(\tau)=\frac{2}{(2\pi i)^k}\left(L(\chi_4,k)+\frac{1}{2^{k}}\til{h}_k(\tau)\right)\ .
\end{equation}
And, for any integers $k_1\geq 2$ and $k_2\geq 3$, we have
\begin{equation}\label{eq:fourier-exp2}
\begin{split}
\til{H}_{k_1,k_2}(\tau)
&=\frac{4}{(2\pi i)^{k_1+k_2}}\Biggl(L_{\sh}(\chi_4,k_1,k_2)
- \frac{1}{2^{k_2}}L(\chi_0,k_1) h_{k_2}(\tau) \\
&\quad -\frac{1}{2^{k_1+k_2}}\sum_{0<m_1<m_2}\til{\vphi}_{k_1}(m_1\tau)\psi_{k_2}(m_2\tau) \\
&\quad\quad +\sum_{j=0}^{k_2-1}\frac{(-1)^{k_1-1}}{2^{k_2-j}}\binom{k_1+j-1}{k_1-1}L(\chi_0,k_1+j)h_{k_2-j}(\tau)\\
&\quad\quad\quad +\sum_{j=0}^{k_1-1}\frac{(-1)^{j}}{2^{k_1-j}}\binom{k_2+j-1}{k_2-1}L(\chi_4,k_2+j)\til{h}_{k_1-j}(\tau)\Biggr).
\end{split}
\end{equation}
Here, for a character $\chi$,
$$L(\chi,k)\ceq \sum_{n=1}^{\infty}\frac{\chi(n)}{n^{k}}\ ,\quad L_{\sh}(\chi,k_1,k_2)\ceq \sum_{0<n_1<n_2}\frac{\chi(n_1)\chi(n_2-n_1)}{n_1^{k_1}n_2^{k_2}}\ .$$
\end{prop}

\begin{proof} Only \cref{eq:fourier-exp2} is shown here. Decompose the given series
$$\sum_{0<4m_1\tau+n_1<4m_2\tau+n_2}\frac{\chi_4(n_1)\chi_4(n_2-n_1)}{(4m_1\tau+n_1)^{k_1}(4m_2\tau+n_2)^{k_2}}
$$
into four terms: $m_1=m_2=0$, $0=m_1<m_2$, $0<m_1=m_2$ and $0<m_1<m_2$. Clearly, The term of $m_1=m_2=0$ is $L_{\sh}(\chi_4,k_1,k_2)$. Also, using the fact that
\begin{equation}\label{eq:char-formal}
\chi_4(n_1)\chi_4(n_2-n_1)=\chi_0(n_1)\chi_4(n_2-1)
\end{equation}
and \cref{eq:vphi-psi}, we obtain 
\begin{alignat*}{5}
&0=m_1<m_2 &\ ; \ & -\frac{1}{2^{k_2}}L(\chi_0,k_1)h_{k_2}(\tau),\\
&0<m_1<m_2 &\ ; \ & -\frac{1}{2^{k_1+k_2}}\sum_{0<m_1<m_2}\til{\vphi}_{k_1}(m_1\tau)\psi_{k_2}(m_2\tau).
\end{alignat*}
Finally, consider the term of $0<m_1=m_2$. Let
$$\Phi_{k_1,k_2}(\tau) \ceq \sum_{\substack{n_1,n_2\in\mathbb{Z}\\ n_1<n_2}}\frac{\chi_4(n_1)\chi_4(n_2-n_1)}{(4\tau+n_1)^{k_1}(4\tau+n_2)^{k_2}}
=\sum_{n_2=1}^{\infty}\sum_{n_1\in\mathbb{Z}}\frac{\chi_4(n_1)\chi_4(n_2)}{(4\tau+n_1)^{k_1}(4\tau+n_1+n_2)^{k_2}}.$$
By \cref{eq:char-formal} and using the partial fractional decomposition formula :
\begin{equation*}
\frac{1}{(\tau+n_1)^{k_1}(\tau+n_1+n_2)^{k_2}}=\sum_{j=0}^{k_2-1}\frac{(-1)^{k_1}\binom{k_1+j-1}{j}}{n_2^{k_1+j}(\tau+n _1+n_2)^{k_2-j}}
+\sum_{j=0}^{k_1-1}\frac{(-1)^{j}\binom{k_2+j-1}{j}}{n_2^{k_2+j}(\tau+n_1)^{k_1-j}},
\end{equation*}
we obtain
\begin{align*}
\Phi_{k_1,k_2}(m\tau)
&=\sum_{j=0}^{k_2-1}(-1)^{k_1}\binom{k_1+j-1}{j}\sum_{n_2=1}^{\infty}\sum_{n_1\in\mathbb{Z}}\frac{\chi_4(n_1+n_2-1)\chi_0(n_2)}{n_2^{k_1+j}(4m\tau+n_1+n_2)^{k_2-j}}
\\
&\quad +\sum_{j=0}^{k_1-1}(-1)^{j}\binom{k_2+j-1}{j}\sum_{n_2=1}^{\infty}\sum_{n_1\in\mathbb{Z}}\frac{\chi_4(n_1)\chi_4(n_2)}{n_2^{k_2+j}(4m\tau+n_1)^{k_1-j}}
\\
&=\sum_{j=0}^{k_2-1}(-1)^{k_1-1}\binom{k_1+j-1}{k_1-1}\frac{1}{2^{k_2-j}}L(\chi_0,k_1+j)\psi_{k_2-j}(m\tau)
\\
&\quad +\sum_{j=0}^{k_1-1}(-1)^{j}\binom{k_2+j-1}{k_2-1}\frac{1}{2^{k_1-j}}L(\chi_4,k_2+j)\til{\psi}_{k_1-j}(m\tau).
\end{align*}
Therefore, by summing over $m$ on both sides, we obtain the Fourier expansion of the term that $0 < m_1 = m_2$. So, the proof is concluded.
\end{proof}

\begin{definition}
For integers $k,k_1,k_2\geq 1$, we define $\til{H}_k(\tau)$ and $\til{H}_{k_1,k_2}(\tau)$ by Fourier expansions \cref{eq:fourier-exp1,eq:fourier-exp2}. Here, we fix a constant $c\in\mathbb{C}$, and set $L(\chi_0,1) = c$.
\end{definition}

\begin{remark}
\begin{enumerate}
\item The definition of $\til{H}_{k_1,k_2}(\tau)$ is independent of the choice of $c$.
\item The constant terms of the Fourier series \cref{eq:fourier-exp1,eq:fourier-exp2} are $\til{T}(k)$ and $\til{T}(k_1,k_2)$ respectively, and these are defined for any $k,k_1,k_2\geq 1$ because $\chi_4$ is a non-trivial character.
\item $\til{T}(k)$ and $\til{T}(k_1,k_2)$ have iterated integral representations:
\begin{equation}\label{eq:iterated-integral}
\begin{split}
\til{T}(k)&=\frac{1}{(2\pi i)^{k}}\int\cdots\int_{0<t_1<\cdots<t_k}\frac{2dt_1}{1+t_1^2}\frac{dt_2}{t_2}\cdots\frac{dt_k}{t_k},
\\
\til{T}(k_1,k_2)&=\frac{1}{(2\pi i)^{k_1+k_2}}\int\cdots\int_{0<t_1<\cdots<t_{k_1+k_2}}\frac{2dt_1}{1+t_1^2}\frac{dt_2}{t_2}\cdots\frac{dt_{k_1}}{t_{k_1}}\frac{2dt_{k_1+1}}{1+t_{k_1+1}^2}\frac{dt_{k_1+2}}{t_{k_1+2}}\cdots\frac{dt_{k_1+k_2}}{t_{k_1+k_2}}.
\end{split}
\end{equation}
\item $\til{H}_k$ and $\til{H}_{k_1,k_2}$ are linear combinations of multiple Eisenstein series of level 4\footnote{Their definition is without the factor $(2\pi i)^{-(k_1+\cdots+k_r)}$. Furthermore, the order of the indices is reversed.} introduced by Yuan and Zhao in \cite{Yuan-Zhao}:
\begin{equation*}
\begin{split}
\til{H}_k(\tau)&=2(G^{1;4}_k(\tau)-G^{3;4}_k(\tau)),
\\
\til{H}_{k_1,k_2}(\tau)&=4(G^{1,2;4}_{k_1,k_2}(\tau)+G^{3,2;4}_{k_1,k_2}(\tau)-G^{1,0;4}_{k_1,k_2}(\tau)-G^{3,0;4}_{k_1,k_2}(\tau)).
\end{split}
\end{equation*}
Therefore, we can show \cref{prop:fourier-exp} from the Fourier expansion of $G^{a_1,a_2;4}_{k_1,k_2}(\tau)$ calculated in \cite[Theorem 4.4]{Yuan-Zhao}.

\item We have
$$4i\til{H}_1(\tau)=1+4\sum_{n=1}^{\infty}\sum_{d|n}\chi_4(n)q^{n}=\theta(\tau)^2\in M_1(\Gamma_0(4),\chi_4)\ ,\quad \left(\theta(\tau)=\sum_{n\in\mathbb{Z}}q^{n^2}\right).$$
\end{enumerate}
\end{remark}

In the following, $B_k$ is the $k$-th Bernoulli number defined by the generating function ($B_1=1/2$):
$$\frac{te^{t}}{e^t-1}=\sum_{n=0}^{\infty}\frac{B_n}{n!}t^n.$$

\begin{lem}\label{lem:multi-fourier-exp}
For any positive integers $k_1$ and $k_2$, we have
\begin{equation}\label{ep:multi-fourier-exp}
\begin{split}
\til{H}_{k_1}(\tau)\til{H}_{k_2}(\tau)
= \frac{4}{(2\pi i)^{k_1+k_2}} \Biggl(
&L(\chi_4,k_1)L(\chi_4,k_2)+\frac{L(\chi_4,k_1)}{2^{k_2}}\til{h}_{k_2}(\tau)+\frac{L(\chi_4,k_2)}{2^{k_1}}\til{h}_{k_1}(\tau) \\
&+\frac{1}{2^{k_1+k_2}}\sum_{\substack{0<m_1,m_2\\m_1\neq m_2}} \til{\psi}_{k_1}(m_1\tau)\til{\psi}_{k_2}(m_2\tau) \ +\  \frac{1}{2^{k_1+k_2}}\til{g}_{k_1+k_2}(\tau) \\
&- \frac{1}{2^{k_1+k_2}}\sum_{j=0}^{k_1-2}(-1)^{k_2}\frac{(-2\pi i)^{k_2+j}}{j!(k_2-1)!}\frac{B_{k_2+j}}{k_2+j} \left(\frac{1}{2^{k_2+j-1}}-1\right) \til{g}_{k_1-j}(\tau) \\
&- \frac{1}{2^{k_1+k_2}}\sum_{j=0}^{k_2-2}(-1)^j \frac{(-2\pi i)^{k_1+j}}{j!(k_1-1)!}\frac{B_{k_1+j}}{k_1+j} \left(\frac{1}{2^{k_1+j-1}}-1\right) \til{g}_{k_2-j}(\tau)
\Biggr)\ .
\end{split}
\end{equation}
\end{lem}

\begin{proof}
From equation \cref{eq:fourier-exp1}, it is enough to prove
\begin{equation*}
\begin{split}
\sum_{0<m_1=m_2}\til{\psi}_{k_1}(m_1\tau)\til{\psi}_{k_2}(m_2\tau)
=\til{g}_{k_1+k_2}(\tau)
&- \sum_{j=0}^{k_1-2}(-1)^{k_2}\frac{(-2\pi i)^{k_2+j}}{j!(k_2-1)!}\frac{B_{k_2+j}}{k_2+j} \left(\frac{1}{2^{k_2+j-1}}-1\right) \til{g}_{k_1-j}(\tau) \\
&- \sum_{j=0}^{k_2-2}(-1)^j \frac{(-2\pi i)^{k_1+j}}{j!(k_1-1)!}\frac{B_{k_1+j}}{k_1+j} \left(\frac{1}{2^{k_1+j-1}}-1\right) \til{g}_{k_2-j}(\tau).
\end{split}
\end{equation*}
By using \cref{eq:char-formal}, we obtain
\begin{align*}
\Psi_{k_1,k_2}(\tau)
&\ceq\sum_{n_1,n_2>0}\chi_4(n_1)\chi_4(n_2)n_1^{k_1-1}n_2^{k_2-1}q^{n_1+n_2}
\\
&=\sum_{n=1}^{\infty}\sum_{w=1}^{n-1}\chi_4(w)\chi_4(n-w)w^{k_1-1}(n-w)^{k_2-1}q^{n}
\\
&=\sum_{n=1}^{\infty}\sum_{w=1}^{n-1}\chi_4(w)\chi_4(n-w)w^{k_1-1}\left(\sum_{i=0}^{k_2-1}\binom{k_2-1}{i}n^{k_2-1-i}(-w)^i\right)q^{n}
\\
&=\sum_{n=1}^{\infty}\sum_{i=0}^{k_2-1}(-1)^i\chi_4(n-1)\binom{k_2-1}{i}n^{k_2-1-i}\left(\sum_{w=1}^{n-1}\chi_0(w)w^{k_1+i-1}\right)q^{n}.
\end{align*}
If $n$ is a positive even integer, Seki-Bernoulli's formula for the sum of powers of integers gives the equation
\begin{align*}
\sum_{w=1}^{n-1}\chi_0(w)w^{k-1}&= \sum_{w=1}^{n}w^{k-1} - \sum_{w=1}^{n/2}(2w)^{k-1} \\
&= \frac{1}{k}\sum_{j=0}^{k-1}\binom{k}{j}B_jn^{k-j}
- 2^{k-1}\frac{1}{k}\sum_{j=0}^{k-1}\binom{k}{j}B_j\left(\frac{n}{2}\right)^{k-j} 
\\
&= \frac{1}{k}\sum_{j=0}^{k-1}\binom{k}{j}B_j(1-2^{j-1})n^{k-j}.
\end{align*}
Thus, we obtain
\begin{align*}
\Psi_{k_1,k_2}(\tau)
&=\sum_{n=1}^{\infty}\sum_{i=0}^{k_2-1}(-1)^i\chi_4(n-1)\binom{k_2-1}{i}n^{k_2-1-i}\left(\frac{1}{k_1+i}\sum_{j=0}^{k_1+i-1}\binom{k_1+i}{j}B_j(1-2^{j-1})n^{k_1+i-j}\right)q^{n}
\\
&=\sum_{i=0}^{k_2-1}\sum_{j=0}^{k_1+i-1}\frac{(-1)^i}{k_1+i}\binom{k_2-1}{i}\binom{k_1+i}{j}B_j(1-2^{j-1})\sum_{n=1}^{\infty}\chi_4(n-1)n^{k_1+k_2-j-1}q^{n}
\\
&=-\sum_{i=0}^{k_2-1}\sum_{j=0}^{k_1+i-1}\frac{(-1)^i}{k_1+i}\binom{k_2-1}{i}\binom{k_1+i}{j}B_j(1-2^{j-1})\frac{(k_1+k_2-j-1)!}{(-2\pi i)^{k_1+k_2-j}}2^{k_1+k_2-j-1}\til{\vphi}_{k_1+k_2-j}(\tau)
\\
&=-\sum_{i=0}^{k_2-1}\sum_{j=0}^{k_1+k_2-1}\frac{(-1)^i}{k_1+i}\binom{k_2-1}{i}\binom{k_1+i}{j}B_j(1-2^{j-1})\frac{(k_1+k_2-j-1)!}{(-2\pi i)^{k_1+k_2-j}}2^{k_1+k_2-j-1}\til{\vphi}_{k_1+k_2-j}(\tau)
\\
&\quad +\sum_{i=0}^{k_2-1}\frac{(-1)^i}{k_1+i}\binom{k_2-1}{i}B_{k_1+i}(1-2^{k_1+i-1})\frac{(k_2-i-1)!}{(-2\pi i)^{k_2-i}}2^{k_2-i-1}\til{\vphi}_{k_2-i}(\tau).
\end{align*}
Furthermore, by using the identity
$$\frac{1}{k+i}\binom{k+i}{j}=\frac{1}{j}\binom{k+i-1}{j-1},$$
we have
\begin{align*}
&\Psi_{k_1,k_2}(\tau)\\
&=-2^{k_1+k_2-2}\sum_{i=0}^{k_2-1}\sum_{j=1}^{k_1+k_2-1}\frac{(-1)^i}{j}\binom{k_2-1}{i}\binom{k_1+i-1}{j-1}B_j\left(\frac{1}{2^{j-1}}-1\right)\frac{(k_1+k_2-j-1)!}{(-2\pi i)^{k_1+k_2-j}}\til{\vphi}_{k_1+k_2-j}(\tau)
\\
&\quad\quad -2^{k_1+k_2-2}\sum_{i=0}^{k_2-1}\frac{(-1)^i}{k_1+i}\binom{k_2-1}{i}\frac{(k_1+k_2-1)!}{(-2\pi i)^{k_1+k_2}}\til{\vphi}_{k_1+k_2}(\tau)
\\
&\quad\quad\quad +2^{k_1+k_2-2}\sum_{i=0}^{k_2-1}\frac{(-1)^i}{k_1+i}\binom{k_2-1}{i}B_{k_1+i}\left(\frac{1}{2^{k_1+i-1}}-1\right)\frac{(k_2-i-1)!}{(-2\pi i)^{k_2-i}}\til{\vphi}_{k_2-i}(\tau).
\end{align*}
On the other hand, by comparing the coefficients of both sides of 
\begin{align*}
\sum_{j=1}^{k_1+k_2-1}\left(\sum_{i=0}^{k_2-1}(-1)^i\binom{k_2-1}{i}\binom{k_1+i-1}{j-1}\right)X^{j-1}
&= \sum_{i=0}^{k_2-1}(-1)^i\binom{k_2-1}{i}(1+X)^{k_1+i-1} \\
&= (1+X)^{k_1-1}(1-(1+X))^{k_2-1} \\
&= (-1)^{k_2-1}\sum_{i=0}^{k_1-1}\binom{k_1-1}{i}X^{k_2-1+i} \\
&= (-1)^{k_2-1}\sum_{j=k_2}^{k_1+k_2-1}\binom{k_1-1}{j-k_2}X^{j-1} ,
\end{align*}
we have
$$
\sum_{i=0}^{k_2-1}(-1)^i\binom{k_2-1}{i}\binom{k_1+i-1}{j-1}=\left\{
\begin{matrix}
\displaystyle{ (-1)^{k_2-1}\binom{k_1-1}{j-k_2} } & (j\geq k_2) \\
0 & (j < k_2) 
\end{matrix}\right. .
$$
Hence, by using the above equation and the equation
$$\sum_{i=0}^{k_2-1}(-1)^{i}\binom{k_2-1}{i}\frac{1}{k_1+i} = \int_{0}^{1}(1-X)^{k_2-1}X^{k_1-1}\ dX = \frac{(k_1-1)!(k_2-1)!}{(k_1+k_2-1)!} ,$$
we obtain
\begin{align*}
\Psi_{k_1,k_2}(\tau)
=&-2^{k_1+k_2-2}\sum_{j=0}^{k_1-1}(-1)^{k_2-1}\binom{k_1-1}{j}\frac{B_{k_2+j}}{k_2+j}\left(\frac{1}{2^{k_2+j-1}}-1\right)\frac{(k_1-j-1)!}{(-2\pi i)^{k_1-j}}\til{\vphi}_{k_1-j}(\tau)
\\
&\quad -2^{k_1+k_2-2}\frac{(k_1-1)!(k_2-1)!}{(-2\pi i)^{k_1+k_2}}\til{\vphi}_{k_1+k_2}(\tau)
\\
&\quad\quad +2^{k_1+k_2-2}\sum_{i=0}^{k_2-1}(-1)^i\binom{k_2-1}{i}\frac{B_{k_1+i}}{k_1+i}\left(\frac{1}{2^{k_1+i-1}}-1\right)\frac{(k_2-i-1)!}{(-2\pi i)^{k_2-i}}\til{\vphi}_{k_2-i}(\tau).
\end{align*}
Note that in the obtained equation, the term $j=k_1-1$ in the first summation and the term $i=k_2-1$ in the second summation cancel each other out. As a result, we obtain the target equation
\begin{align*}
\sum_{0<m_1=m_2}\til{\psi}_{k_1}(m_1\tau)\til{\psi}_{k_2}(m_2\tau)
&=\frac{(-2\pi i)^{k_1+k_2}}{(k_1-1)!(k_2-1)!}\frac{-1}{2^{k_1+k_2-2}}\sum_{m=1}^{\infty}\Psi_{k_1,k_2}(m\tau)
\\
&=\til{g}_{k_1+k_2}(\tau)
-\sum_{j=0}^{k_1-2}(-1)^{k_2}\frac{(-2\pi i)^{k_2+j}}{j!(k_2-1)!}\frac{B_{k_2+j}}{k_2+j}\left(\frac{1}{2^{k_2+j-1}}-1\right)\til{g}_{k_1-j}(\tau)
\\
&\quad -\sum_{j=0}^{k_2-2}(-1)^j\frac{(-2\pi i)^{k_1+j}}{j!(k_1-1)!}\frac{B_{k_1+j}}{k_1+j}\left(\frac{1}{2^{k_1+j-1}}-1\right)\til{g}_{k_2-j}(\tau).
\end{align*}
\end{proof}

\section{Proof of \cref{thm:shuffle}}\label{sec:proof1}
\subsection{Proof of the shuffle relation \cref{eq:shuffle}}\label{sec:shuffle}

\begin{lem}\label{lem:sub-lemma}
For non-negative integers $\mu,\alpha,\beta$ such that $\mu\leq \alpha+\beta$,
\begin{equation*}
\sum_{\nu=0}^{\mu}(-1)^{\nu}\binom{\alpha+\beta-\nu}{\alpha}\binom{\mu}{\nu}=\binom{\alpha+\beta-\mu}{\beta}.
\end{equation*}
\end{lem}

\begin{proof} We compute the generating functions of both sides:
\begin{align*}
&\sum_{\alpha=0}^{\infty}\sum_{\beta=0}^{\infty}\sum_{\mu=0}^{\alpha+\beta}\sum_{\nu=0}^{\mu}(-1)^{\nu}\binom{\alpha+\beta-\nu}{\alpha}\binom{\mu}{\nu}X^{\alpha}Y^{\beta}Z^{\mu}
\\
&=\sum_{N=0}^{\infty}\sum_{\alpha=0}^{N}\sum_{\mu=0}^{N}\sum_{\nu=0}^{\mu}(-1)^{\nu}\binom{N-\nu}{\alpha}\binom{\mu}{\nu}X^{\alpha}Y^{N-\alpha}Z^{\mu}\quad (N=\alpha+\beta)
\\
&=\sum_{N=0}^{\infty}\sum_{\mu=0}^{N}\sum_{\nu=0}^{\mu}(-1)^{\nu}\binom{\mu}{\nu}(X+Y)^{N-\nu}Y^{\nu}Z^{\mu}
\\
&=\sum_{N=0}^{\infty}\sum_{\mu=0}^{N}(X+Y)^{N}\left(1-\frac{Y}{X+Y}\right)^{\mu}Z^{\mu}
\\
&=\sum_{N=0}^{\infty}\sum_{\mu=0}^{N}(X+Y)^{N-\mu}X^{\mu}Z^{\mu}
=\frac{1}{1-(X+Y)}\frac{1}{1-XZ},
\end{align*}

\begin{align*}
&\sum_{\alpha=0}^{\infty}\sum_{\beta=0}^{\infty}\sum_{\mu=0}^{\alpha+\beta}\binom{\alpha+\beta-\mu}{\beta}X^{\alpha}Y^{\beta}Z^{\mu}
\\
&=\sum_{N=0}^{\infty}\sum_{\beta=0}^{N}\sum_{\mu=0}^{N}\binom{N-\mu}{\beta}X^{N-\beta}Y^{\beta}Z^{\mu}\quad (N=\alpha+\beta)
\\
&=\sum_{N=0}^{\infty}\sum_{\mu=0}^{N}X^{\mu}(X+Y)^{N-\mu}Z^{\mu}
=\frac{1}{1-(X+Y)}\frac{1}{1-XZ}.
\end{align*}
\end{proof}

\begin{lem}\label{lem:lem-suffle}
Let $k_1$, $k_2$ be positive integers, and let $\{a_i\}_i$, $\{b_i\}_i$ be arbitrary sequences. Then
\begin{align}
\sum_{i=0}^{k_2-1}\binom{k_1+i-1}{k_1-1}\sum_{j=0}^{k_1+i-1}(-1)^{k_2-i}\binom{k_2-i+j-1}{j}a_{k_2-i+j}b_{k_1+i-j}&=-\sum_{i=0}^{k_1-1}\binom{k_2+i-1}{k_2-1}a_{k_1-i}b_{k_2+i}, \label{eq:lem1}
\\
\sum_{i=0}^{k_2-1}\binom{k_1+i-1}{k_1-1}\sum_{j=0}^{k_2-i-1}(-1)^{j}\binom{k_1+i+j-1}{j}a_{k_1+i+j}b_{k_2-i-j}&=a_{k_1}b_{k_2}. \label{eq:lem2}
\end{align}
\end{lem}

\begin{proof}
First, we prove \cref{eq:lem2}. Set $i+j=N$, then
\begin{align*}
&\sum_{i=0}^{k_2-1}\binom{k_1+i-1}{k_1-1}\sum_{j=0}^{k_2-i-1}(-1)^{j}\binom{k_1+i+j-1}{j}a_{k_1+i+j}b_{k_2-i-j}
\\
&=\sum_{N=0}^{k_2-1}\sum_{i=0}^{N}\binom{k_1+i-1}{k_1-1}(-1)^{N-i}\binom{k_1+N-1}{N-i}a_{k_1+N}b_{k_2-N}
\\
&=\sum_{N=0}^{k_2-1}(-1)^N\frac{(k_1+N-1)!}{(k_1-1)!N!}\sum_{i=0}^{N}(-1)^i\binom{N}{i}a_{k_1+N}b_{k_2-N}
\\
&=a_{k_1}b_{k_2}.
\end{align*}
Therefore, the proof of \cref{eq:lem2} is complete. Next, we prove \cref{eq:lem1}. We have
\begin{eqnarray*}
&&\sum_{i=0}^{k_2-1}\binom{k_1+i-1}{k_1-1}\sum_{j=0}^{k_1+i-1}(-1)^{k_2-i}\binom{k_2-i+j-1}{j}a_{k_2-i+j}b_{k_1+i-j}
\\
&\overset{k_2-i-1=v}{=}&\sum_{v=0}^{k_2-1}\binom{k_1+k_2-v-2}{k_1-1}\sum_{j=0}^{k_1+k_2-v-2}(-1)^{1+v}\binom{v+j}{j}a_{v+j+1}b_{k_1+k_2-v-j-1}
\\
&=&\sum_{v=0}^{k_1+k_2-2}\binom{k_1+k_2-v-2}{k_1-1}\sum_{j=0}^{k_1+k_2-v-2}(-1)^{1+v}\binom{v+j}{v}a_{v+j+1}b_{k_1+k_2-v-j-1}
\\
&\overset{v+j=\mu}{=}&
\sum_{v=0}^{k_1+k_2-2}\sum_{\mu=v}^{k_1+k_2-2}\binom{k_1+k_2-v-2}{k_1-1}(-1)^{1+v}\binom{\mu}{v}a_{\mu+1}b_{k_1+k_2-\mu-1}
\\
&\overset{\text{\cref{lem:sub-lemma}}}{=}&
-\sum_{\mu=0}^{k_1+k_2-2}\binom{k_1+k_2-2-\mu}{k_2-1}a_{\mu+1}b_{k_1+k_2-\mu-1}
\\
&\overset{i=k_1-1-\mu}{=}&
-\sum_{i=0}^{k_1-1}\binom{k_2+i-1}{k_2-1}a_{k_1-i}b_{k_2+i}.
\end{eqnarray*}
Therefore, the proof of \cref{eq:lem1} is complete.
\end{proof}

\begin{lem}\label{lem:lemma-phi-phi}
For integers $r,s\geq 1$, let
\begin{equation*}
\Omega_{r,s}(\tau):=\sum_{0<m_1<m_2}\til{\vphi}_{r}(m_1\tau)\psi_{s}(m_2\tau).
\end{equation*}
Then, for integers $k_1,k_2\geq 1$, 
\begin{align*}
\til{h}_{k_1}(\tau)\til{h}_{k_2}(\tau)
&=-\left(\sum_{i=0}^{k_2-1}\binom{k_1+i-1}{k_1-1}\Omega_{k_2-i,k_1+i}(\tau)
+\sum_{i=0}^{k_1-1}\binom{k_2+i-1}{k_2-1}\Omega_{k_1-i,k_2+i}(\tau)\right).
\end{align*}
\end{lem}

\begin{proof}
Clearly, by definition of $\Omega_{r,s}$,
\begin{equation*}
\begin{split}
\Omega_{k_2-i,k_1+i}(\tau)
&=\frac{(-2\pi i)^{k_1+k_2}}{(k_2-i-1)!(k_1+i-1)!}\sum_{0<m_1<m_2}\Biggl(
\frac{1}{2}\delta_{k_2-i,1}\frac{1}{2^{k_1+i-1}}\sum_{n_2=1}^{\infty}\chi_0(n_2)n_2^{k_1+i-1}q^{n_2m_2}
\\
&\quad +\frac{1}{2^{k_1+k_2-2}}\sum_{0<n_1,n_2}\chi_4(n_1+1)\chi_0(n_2)n_1^{k_2-i-1}n_2^{k_1+i-1}q^{n_1m_1+n_2m_2}\Biggr).
\end{split}
\end{equation*}
Thus, by using equation
$$\frac{\binom{k_1+i-1}{k_1-1}}{(k_2-i-1)!(k_1+i-1)!} = \frac{\binom{k_2-1}{i}}{(k_2-1)!(k_1-1)!},$$
we obtain 
\begin{equation*}
\begin{split}
&\sum_{i=0}^{k_2-1}\binom{k_1+i-1}{k_1-1}\Omega_{k_2-i,k_1+i}(\tau)
\\
&\quad=\frac{(-2\pi i)^{k_1+k_2}}{(k_1-1)!(k_2-1)!}\sum_{0<m_1<m_2}\Biggl(
\frac{1}{2}\frac{1}{2^{k_1+k_2-2}}\sum_{n_2=1}^{\infty}\chi_0(n_2)n_2^{k_1+k_2-2}q^{n_2m_2}
\\
&\quad\quad +\frac{1}{2^{k_1+k_2-2}}\sum_{0<n_1,n_2}\chi_4(n_1+1)\chi_0(n_2)(n_1+n_2)^{k_2-1}n_2^{k_1-1}q^{n_1m_1+n_2m_2}\Biggr)
\\
&\quad=\frac{(-2\pi i)^{k_1+k_2}}{(k_1-1)!(k_2-1)!}\sum_{0<m_1,m_2}\Biggl(
\frac{1}{2}\frac{1}{2^{k_1+k_2-2}}\sum_{n=1}^{\infty}(\chi_4(n))^2n^{k_1+k_2-2}q^{n(m_1+m_2)}
\\
&\quad\quad +\frac{1}{2^{k_1+k_2-2}}\sum_{0<n_1<n_2}\chi_4(n_1)\chi_4(n_2)n_1^{k_1-1}n_2^{k_2-1}q^{n_2m_1+n_1m_2}\Biggr).
\end{split}
\end{equation*}
Similarly, calculate the other term in the equation to be proven, and then take the sum of these terms. We obtain this lemma.
\end{proof}

\begin{proof}[\textbf{Proof of the shuffle relation \cref{eq:shuffle}}]

The constant terms $\til{T}(k)$ and $\til{T}(k_1,k_2)$ of the Fourier series \cref{eq:fourier-exp1,eq:fourier-exp2} for $\til{H}_k$ and $\til{H}_{k_1,k_2}$ satisfy the shuffle relation because they have the iterated integral representation \cref{eq:iterated-integral}. By \cref{lem:lem-suffle},
\begin{eqnarray*}
&&\sum_{i=0}^{k_2-1}\binom{k_1+i-1}{k_1-1}\Biggl(-\frac{1}{2^{k_2-i}}L(\chi_0,k_2-i)h_{k_1+i}(\tau)
\\
&&\quad +\sum_{j=0}^{k_1+i-1}\frac{(-1)^{k_2-i-1}}{2^{k_1+i-j}}\binom{k_2-i+j-1}{k_2-i-1}L(\chi_0,k_2-i+j)h_{k_1+i-j}(\tau)
\\
&&\quad\quad +\sum_{j=0}^{k_2-i-1}\frac{(-1)^{j}}{2^{k_2-i-j}}\binom{k_1+i+j-1}{k_1+i-1}L(\chi_4,k_1+i+j)\til{h}_{k_2-i-j}(\tau)\Biggr)
\\
&=&-\sum_{i=0}^{k_2-1}\binom{k_1+i-1}{k_1-1}\frac{1}{2^{k_2-i}}L(\chi_0,k_2-i)h_{k_1+i}(\tau)
\\
&&\quad +\sum_{j=0}^{k_1-1}\binom{k_2+j-1}{k_2-1}\frac{1}{2^{k_2+j}}L(\chi_0,k_1-j)h_{k_2+j}(\tau)
\\
&&\quad\quad +\frac{1}{2^{k_2}}L(\chi_4,k_1)\til{h}_{k_2}(\tau).
\end{eqnarray*}
Therefore, by the above calculation, the shuffle relation at the constant terms and \cref{lem:lemma-phi-phi}, we obtain
\begin{equation*}
\begin{split}
&\sum_{p=1}^{k_1+k_2-1}\left(\binom{p-1}{k_1-1}+\binom{p-1}{k_2-1}\right)\til{H}_{k_1+k_2-p,p}(\tau)
\\
&=\sum_{i=0}^{k_2-1}\binom{k_1+i-1}{k_1-1}\til{H}_{k_2-i,k_1+i}(\tau)
+\sum_{i=0}^{k_1-1}\binom{k_2+i-1}{k_2-1}\til{H}_{k_1-i,k_2+i}(\tau)
\\
&=\frac{4}{(2\pi i)^{k_1+k_2}}\left(L(\chi_4,k_1)+\frac{1}{2^{k_1}}\til{h}_{k_1}(\tau)\right)\left(L(\chi_4,k_2)+\frac{1}{2^{k_2}}\til{h}_{k_2}(\tau)\right).
\end{split}
\end{equation*}
The proof is complete.
\end{proof}

\subsection{Proofs of \cref{eq:eisen1,eq:eisen2}}\label{sec:eisen}
We define the generating functions by
$$H_k(X,Y)\ceq\sum_{\substack{r+s=k\\r,s\geq 1}}\til{H}_{r,s}X^{r-1}Y^{s-1}\ ,\quad B_k(X,Y)\ceq\sum_{\substack{r+s=k\\r,s\geq 1}}\til{H}_{r}\til{H}_{s}X^{r-1}Y^{s-1}.$$
Then the shuffle relation \cref{eq:shuffle} is equivalent to
\begin{equation}\label{eq:formal-shuffle}
H_k(X,X+Y)+H_k(Y,X+Y)=B_k(X,Y).
\end{equation}
Also equation \cref{eq:eisen1} (the equation to be proved) is equivalent to
\begin{equation}\label{eq:formal-eisen1}
\frac{1}{2(k-1)}\left(H_k(X,2X)+H_k(X,0)\right)=\til{G}_kX^{k-2}.
\end{equation}
By substituting $X=Y$ and $X=-Y$ into \cref{eq:formal-shuffle}, we obtain
$$H_k(X,2X)+H_k(X,2X)=B_k(X,X)\ ,\quad H_k(X,0)+H_k(-X,0)=B_k(X,-X).$$
Thus, if $k$ is even, we obtain
\begin{equation*}
H_k(X,2X)+H_k(X,0)
=\frac{1}{2}B_k(X,X)+\frac{1}{2}B_k(X,-X)
=\sum_{r=1:\text{odd}}^{k-1}\til{H}_{k-r}\til{H}_{r}X^{k-2}.
\end{equation*}
Therefore, the right-hand side of \cref{eq:eisen1} equals the right-hand side of \cref{eq:eisen2}. Next, we show that the constant terms of Fourier expansions on both sides of \cref{eq:eisen2} coincide. The Fourier expansion of $\til{G}_{k}$ is the form of 
\begin{equation}
\til{G}_{k}=\frac{1}{(2\pi i)^{k}}\left(L(\chi_0,k)+\frac{1}{2^{k}}\til{g}_k(\tau)\right).
\end{equation}
The constant terms of $\til{H}_{k}$ and $\til{G}_k$ are $2(2\pi i)^{-k}L(\chi_4,k)$ and $(2\pi i)^{-k}L(\chi_0,k)$ respectively, and 
\begin{alignat}{4}
L(\chi_0,2k)&=(1-2^{-2k})\zeta(2k)=-(1-2^{-2k})\frac{(2\pi i)^{2k}}{2}\frac{B_{2k}}{(2k)!}&\quad& (k\geq 1) ,\label{eq:zeta-Bernoulli}
\\
L(\chi_4,2k+1)&=\frac{\pi^{2k+1}}{2^{2k+2}}\frac{(-1)^kE_{2k}}{(2k)!}=\frac{-i}{2}\frac{(\pi i)^{2k+1}}{2^{2k+1}}\frac{E_{2k}}{(2k)!}&&(k\geq 0), \label{eq:L-Euler}
\end{alignat}
where $E_k$ is the $k$-th Euler number defined by the generating function 
$$\frac{2}{e^t+e^{-t}}=\sum_{k=0}^{\infty}E_k\frac{t^k}{k!}.$$
Therefore, to prove that the constant terms on both sides of \cref{eq:eisen2} coincide, it is sufficient to prove \cref{lem:eisen-const} below.

\begin{lem}\label{lem:eisen-const} For any integer $k\geq 2$,
\begin{equation*}
\frac{B_k}{k}\left(1-\frac{1}{2^k}\right)
=\frac{1}{4^k}\sum_{r=1}^{k-1}\binom{k-2}{r-1}E_{k-r-1}E_{r-1}.
\end{equation*}
\end{lem}

\begin{proof}
Let $f(t)$ and $g(t)$ be 
$$f(t)\ceq \frac{te^t}{e^t-1}-\frac{t}{2}-1=\sum_{k=2}B_k\frac{t^{k}}{k!}
\ ,\quad g(t)\ceq \frac{2}{e^{t/4}+e^{-t/4}}=\sum_{k=0}^{\infty}\frac{E_k}{4^k}\frac{t^k}{k!}.$$
Then
\begin{align*}
\frac{t^2}{4^2}g(t)^2
&=\frac{t^2}{4^2}\sum_{k=0}^{\infty}\sum_{r=0}^{k}\frac{E_{k-r}E_{r}}{(k-r)!r!}\frac{t^k}{4^k}
\\
&=\frac{t^2}{4^2}\sum_{k=0}^{\infty}\frac{1}{4^k}\sum_{r=0}^{k}\binom{k}{r}E_{k-r}E_{r}\frac{t^k}{k!}
\\
&=\sum_{k=2}^{\infty}\frac{1}{4^{k}}\sum_{r=1}^{k-1}\binom{k-2}{r-1}E_{k-r-1}E_{r-1}\frac{t^{k}}{(k-2)!},
\end{align*}
and
$$
\til{f}(t)
\ceq t^2\frac{d}{dt}\left(\frac{1}{t}f(t)\right)
= t^2\frac{d}{dt}\left(\sum_{k=2}^{\infty}B_k\frac{t^{k-1}}{k!}\right)
= \sum_{k=2}^{\infty}\frac{1}{k}B_k\frac{t^{k}}{(k-2)!}.
$$
Thus, we show this lemma if we show
$$\frac{t^2}{4^2}g(t)^2=\til{f}(t)-\til{f}(t/2),$$
and it is easy to verify by the equations
$$\frac{t^2}{4^2}g(t)^2 = \frac{(t/2)^2e^{t/2}}{(e^{t/2}+1)^2}
\ ,\quad
\til{f}(t)
=t^2\frac{d}{dt}\left(\frac{e^t}{e^t-1}-\frac{1}{2}-\frac{1}{t}\right)
=-\frac{t^2e^t}{(e^t-1)^2}+1.
$$
\end{proof}

Since we have established that the constant terms on both sides of \cref{eq:eisen2} coincide, it suffices to show that the functions excluding the constant terms on both sides of \cref{eq:eisen1} coincide.

Let $\displaystyle{\til{G}_k^{0}(\tau)\ceq \frac{1}{(2\pi i)^{k}}\frac{1}{2^{k}}\til{g}_k(\tau)}$ (excluded the constant term from $\til{G}_{k}$ if $k$ is even) and 
\begin{equation*}
H_k^0(X,Y)\ceq\sum_{\substack{r+s=k\\r,s\geq 1}}\til{H}_{r,s}^0X^{r-1}Y^{s-1},
\end{equation*}
where $\til{H}_{r,s}^0$ are functions obtained by removing the constant term from $\til{H}_{r,s}$.

\begin{lem}\label{lem:hatH}
For $k_1,k_2\geq 1$, we define $\wh{H}_{k_1,k_2}^0$ by
\begin{align*}
\wh{H}_{k_1,k_2}^0(\tau)
&\ceq \frac{4}{(2\pi i)^{k_1+k_2}}\Biggl(
\frac{1}{2^{k_2}}L(\chi_4,k_1) \til{h}_{k_2}(\tau) \\
&+\frac{1}{2^{k_1+k_2}}\sum_{0<m_1<m_2}\til{\psi}_{k_1}(m_1\tau)\til{\psi}_{k_2}(m_2\tau) \\
&\quad +\frac{1}{2^{k_1+k_2}}\sum_{j=0}^{k_2-2}(-1)^{k_1}\left(\frac{1}{2^{k_1+j-1}}-1\right)\zeta(k_1+j)\binom{k_1+j-1}{j}\til{g}_{k_2-j}(\tau) \\
& \quad\quad +\frac{1}{2^{k_1+k_2}}\sum_{j=0}^{k_1-2}(-1)^{j}\left(\frac{1}{2^{k_2+j-1}}-1\right)\zeta(k_2+j)\binom{k_2+j-1}{j}\til{g}_{k_1-j}(\tau)\Biggr),
\end{align*}
where let $\zeta(1)=0$. (It is sufficient for $\zeta(1)$ to be finite.) And let 
\begin{equation*}
\wh{H}_k^0(X,Y)\ceq\sum_{\substack{r+s=k\\r,s\geq 1}}\wh{H}_{r,s}^0X^{r-1}Y^{s-1}.
\end{equation*}
Then, we have
\begin{equation}\label{formal-til-hat}
H_k^0(X,X+Y)+H_k^0(Y,X+Y)=\wh{H}_{k}^0(X,Y)+\wh{H}_{k}^0(Y,X)+4\til{G}_k^0\frac{X^{k-1}-Y^{k-1}}{X-Y}.
\end{equation}
\end{lem}
\begin{proof}
By \cref{lem:multi-fourier-exp} and simple calculations.
\end{proof}

\begin{remark}
$\wh{H}_{k_1,k_2}^0(\tau)$ is induced by Fourier expansion of the function
\begin{equation*}
\frac{4}{(2\pi i)^{k_1+k_2}}\sum_{0\prec 4m_1\tau+n_1\prec 4m_2\tau+n_2}\frac{\chi_4(n_1)\chi_4(n_2)}{(4m_1+n_1)^{k_1}(4m_2+n_2)^{k_2}}.
\end{equation*}
\end{remark}

By substituting X = Y into \cref{formal-til-hat}, we abtain
\begin{equation}\label{eq:HtoG}
\frac{1}{2(k-1)}\left(H_k^0(X,2X)-\wh{H}_{k}^0(X,X)\right)=\til{G}_k^0X^{k-2}.
\end{equation}
Upon comparing \cref{eq:formal-eisen1} and \cref{eq:HtoG}, we can establish the validity of \cref{eq:eisen1,eq:eisen2} by proving
\begin{equation*}
H_k^0(X,0)+\wh{H}_{k}^0(X,X)=\left(\til{H}_{k-1,1}^0+\sum_{\substack{r+s=k\\r,s\geq 1}}\wh{H}_{r,s}^0\right)X^{k-2}=0.
\end{equation*}

\begin{prop}\label{H+H=0} For any even integer $k\geq 4$, we have
\begin{equation}
\til{H}_{k-1,1}^0+\sum_{\substack{r+s=k\\r,s\geq 1}}\wh{H}_{r,s}^0=0.
\end{equation}
\end{prop}

\begin{proof}
Through a simple calculation, we obtain
\begin{align*}
\til{H}_{k-1,1}^0+\sum_{\substack{r+s=k\\r,s\geq 1}}\wh{H}_{r,s}^0=\frac{1}{2^{k}}\frac{4}{(2\pi i)^{k}}\Biggl(
&\sum_{r=1:\text{odd}}^{k-1}2^{r+1} L_*(\chi_4,r) \til{h}_{k-r}(\tau)
\\
&+\sum_{r=1}^{k-1}\sum_{0<m_1<m_2}\til{\psi}_{r}(m_1\tau)\til{\psi}_{k-r}(m_2\tau)
-\sum_{0<m_1<m_2}\psi_{1}(m_2\tau)\til{\vphi}_{k-1}(m_1\tau)\Biggr).
\end{align*}
By using the identity
$$\frac{q^{n_1}}{1-q^{n_1}}\frac{q^{n_2}}{1-q^{n_2}}=\frac{q^{n_1+n_2}}{1-q^{n_1+n_2}}+\frac{q^{n_1}}{1-q^{n_1}}\frac{q^{n_1+n_2}}{1-q^{n_1+n_2}}+\frac{q^{n_2}}{1-q^{n_2}}\frac{q^{n_1+n_2}}{1-q^{n_1+n_2}},$$
we heve
\begin{align*}
&\sum_{0<m_1<m_2}\psi_{1}(m_2\tau)\til{\vphi}_{k-1}(m_1\tau)
\\
&=\frac{(-2\pi i)^{k}}{(k-2)!}\frac{1}{2^{k-2}}\sum_{0<m_1<m_2}\sum_{n_1,n_2=1}^{\infty}\chi_0(n_2)\chi_4(n_1+1)n_1^{k-2}q^{n_1m_1+n_2m_2}
\\
&=\frac{(-2\pi i)^{k}}{(k-2)!}\frac{1}{2^{k-2}}\sum_{n_1,n_2=1}^{\infty}\chi_0(n_2)\chi_4(n_1+1)n_1^{k-2}\frac{q^{n_1+n_2}}{1-q^{n_1+n_2}}\frac{q^{n_2}}{1-q^{n_2}}
\\
&=\frac{(-2\pi i)^{k}}{(k-2)!}\frac{1}{2^{k-2}}\sum_{n_1,n_2=1}^{\infty}\chi_0(n_2)\chi_4(n_1+1)n_1^{k-2}\frac{q^{n_1}}{1-q^{n_1}}\frac{q^{n_2}}{1-q^{n_2}}
\\
&\quad -\frac{(-2\pi i)^{k}}{(k-2)!}\frac{1}{2^{k-2}}\sum_{n_1,n_2=1}^{\infty}\chi_0(n_2)\chi_4(n_1+1)n_1^{k-2}\left(\frac{q^{n_1+n_2}}{1-q^{n_1+n_2}}+\frac{q^{n_1}}{1-q^{n_1}}\frac{q^{n_1+n_2}}{1-q^{n_1+n_2}}\right).
\end{align*}
And, we have
\begin{align*}
\sum_{r=1}^{k-1}\sum_{0<m_1<m_2}\til{\psi}_{r}(m_1\tau)\til{\psi}_{k-r}(m_2\tau) 
&=-\frac{(-2\pi i)^{k}}{(k-2)!}\frac{1}{2^{k-2}}\sum_{0<m_1<m_2}\sum_{n_1,n_2=1}^{\infty}\chi_4(n_1)\chi_4(n_2)(n_1+n_2)^{k-2}q^{n_1m_1+n_2m_2}
\\
&=\frac{(-2\pi i)^{k}}{(k-2)!}\frac{1}{2^{k-2}}\sum_{0<n_2<n_1}\chi_0(n_2)\chi_4(n_1+1)n_1^{k-2}\frac{q^{n_1}}{1-q^{n_1}}\frac{q^{n_2}}{1-q^{n_2}}.
\end{align*}
Therefore, we have
\begin{align*}
&\sum_{r=1}^{k-1}\sum_{0<m_1<m_2}\til{\psi}_{r}(m_1\tau)\til{\psi}_{k-r}(m_2\tau) 
-\sum_{0<m_1<m_2}\psi_{1}(m_2\tau)\til{\vphi}_{k-1}(m_1\tau)
\\
&=\frac{(-2\pi i)^{k}}{(k-2)!}\frac{1}{2^{k-2}}\sum_{0<n_1,n_2}\chi_0(n_2)\chi_4(n_1+1)n_1^{k-2}\left(\frac{q^{n_1+n_2}}{1-q^{n_1+n_2}}+\frac{q^{n_1}}{1-q^{n_1}}\frac{q^{n_1+n_2}}{1-q^{n_1+n_2}}\right)
\\
&\quad -\frac{(-2\pi i)^{k}}{(k-2)!}\frac{1}{2^{k-2}}\sum_{0<n_1<n_2}\chi_0(n_2)\chi_4(n_1+1)n_1^{k-2}\frac{q^{n_1}}{1-q^{n_1}}\frac{q^{n_2}}{1-q^{n_2}}
\\
&=\frac{(-2\pi i)^{k}}{(k-2)!}\frac{1}{2^{k-2}}\sum_{0<n_1,n_2}\chi_0(n_2)\chi_4(n_1+1)n_1^{k-2}\frac{q^{n_1+n_2}}{1-q^{n_1+n_2}}
\\
&\quad +\frac{(-2\pi i)^{k}}{(k-2)!}\frac{1}{2^{k-2}}\sum_{0<n_1,n_2}(\chi_0(n_2)-\chi_0(n_1+n_2))\chi_4(n_1+1)n_1^{k-2}\frac{q^{n_1}}{1-q^{n_1}}\frac{q^{n_1+n_2}}{1-q^{n_1+n_2}}
\\
&=\frac{(-2\pi i)^{k}}{(k-2)!}\frac{1}{2^{k-2}}\sum_{0<n_1,n_2}\chi_0(n_2)\chi_4(n_1+1)n_1^{k-2}\frac{q^{n_1+n_2}}{1-q^{n_1+n_2}}
\\
&=\frac{(-2\pi i)^{k}}{(k-2)!}\frac{1}{2^{k-2}}\sum_{n=1}^{\infty}\chi_0(n)\sum_{w=1}^{n-1}\chi_4(w+1)w^{k-2}\frac{q^{n}}{1-q^{n}}.
\end{align*}
On the other hand, by \cref{eq:L-Euler}, we have
\begin{align*}
\sum_{r=1:\text{odd}}^{k-1}2^{r+1} L(\chi_4,r) \til{h}_{k-r}(\tau)
&=\sum_{r=1:\text{odd}}^{k-1}2^{r+1} \frac{-i}{2}\frac{(\pi i)^{r}}{2^{r}}\frac{E_{r-1}}{(r-1)!} \frac{(-2\pi i)^{k-r}}{(k-r-1)!}\frac{i}{2^{k-r-1}}\sum_{m=1}^{\infty}\sum_{n=1}^{\infty}\chi_4(n)n^{k-r-1}q^{mn}
\\
&=-\frac{1}{2^{k-1}}\frac{(-2\pi i)^{k}}{(k-2)!}\sum_{r=1:\text{odd}}^{k-1}\binom{k-2}{r-1}E_{r-1}\sum_{n=1}^{\infty}\chi_4(n)n^{k-r-1}\frac{q^{n}}{1-q^{n}}
\\
&=-\frac{1}{2}\frac{(-2\pi i)^{k}}{(k-2)!}\sum_{n=1}^{\infty}\chi_4(n)E_{k-2}\left(\frac{n+1}{2}\right)\frac{q^{n}}{1-q^{n}}.
\end{align*}
Here, $E_k(x)$ is the Euler polynomial defined by
$$E_k(x):=\sum_{n=0}^{k}\binom{k}{n}\frac{E_n}{2^n}\left(x-\frac{1}{2}\right)^{k-n},$$
and it has a generating function
$$\sum_{k=0}^{\infty}\frac{E_k(x)}{k!}t^k=\frac{2e^{xt}}{e^t+1}.$$
Therefore, for any even integer $k\geq 4$ and any $n\geq 1$, it is enough to show
\begin{align*}
\chi_0(n)\sum_{w=1}^{n-1}\chi_4(w+1)\left(\frac{w}{2}\right)^{k-2}
&=\frac{\chi_4(n)}{2}E_{k-2}\left(\frac{n+1}{2}\right).
\end{align*}
This equation can be verified by \cref{lem:euler} below.
\end{proof}

\begin{lem}\label{lem:euler} For any integer $k\geq 1$ and any odd integer $n\geq 1$, we have
\begin{align}
E_{k}(1-x) - (-1)^{k} E_{k}(x)&=0,\label{eq:e1}
\\
\sum_{w=1}^{n-1}\chi_4(w+1)\left(\frac{w}{2}\right)^{k}
&=\frac{\chi_4(n)}{2}E_{k}\left(\frac{n+1}{2}\right) - \frac{1}{4}(E_{k}(1) - E_{k}(0)).\label{eq:e2}
\end{align}
\end{lem}

\begin{proof}
Equation \cref{eq:e1} is well known. We only prove \cref{eq:e2}. By multiplying both sides by $X^{k}/k!$ and summing up over $k$ from $1$ to $\infty$, \cref{eq:e2} is equivalent to 
\begin{equation*}
\sum_{w=1}^{n-1}\chi_4(w+1)e^{\frac{wX}{2}}
-\sum_{w=1}^{n-1}\chi_4(w+1)
=\chi_4(n)\frac{e^{\frac{n}{2}X}}{e^{\frac{X}{2}}+e^{-\frac{X}{2}}}
-\frac{\chi_4(n)}{2}-\frac{1}{2}\frac{e^{\frac{X}{2}}-e^{-\frac{X}{2}}}{e^{\frac{X}{2}}+e^{-\frac{X}{2}}}.
\end{equation*}
By using the identity
$$\frac{\chi_4(n)}{2}-\sum_{w=1}^{n-1}\chi_4(w+1)=\frac{1}{2}\quad (n\geq 1:\text{odd}),$$
multiplying both sides by $e^{\frac{X}{2}}+e^{-\frac{X}{2}}$ and defining $f(w):=\chi_4(w)e^{\frac{wX}{2}}$, this is equivalent to
\begin{equation*}
\sum_{w=1}^{n-1}(f(w+1)-f(w-1))+\frac{1}{2}(f(1)-f(-1))
=f(n)-\frac{1}{2}(f(1)+f(-1)),
\end{equation*}
which is clear from $f(0)=f(n-1)=0$ (because $n$ is odd).
\end{proof}

Therefore, we proved \cref{eq:eisen1,eq:eisen2}.

\section{The dimension of $\Im\DHk$}\label{sec:ImDH}
Let $f(\tau)$ be a complex function with a Fourier expansion of the form $f(\tau)=\displaystyle{\sum_{n\in\mathbb{Z}}(x_n+iy_n)q^n},\ (x_n,y_n\in\mathbb{R})$. Then we define maps $\Im_n$ and $\Im$ as
$$\Im_n(f)= i y_n\ ,\quad \Im(f)(\tau)= \sum_{n>0}iy_nq^n .$$
The goal of this section is to determine the dimension of $\Im\DHk$ as $\mathbb{Q}$-vector space, where $\Im\DHk$ is the image of $\DHk$ under the linear map $\Im$.

We assume $k$ is an even integer. Then
\begin{equation}\label{eq:ImH_even}
\begin{split}
\Im(\til{H}_{r,k-r})(\tau)
&=\frac{4}{(2\pi i)^{k}}\Biggl(\sum_{j=0}^{k-r-1}\frac{(-1)^{r-1}}{2^{k-r-j}}\binom{r+j-1}{r-1}\delta_{k-r-j,\text{od}}(1-\delta_{j,0})L(\chi_0,r+j)h_{k-r-j}(\tau)\\
&\quad\quad +\sum_{j=0}^{r-1}\frac{(-1)^{j}}{2^{r-j}}\binom{k-r+j-1}{k-r-1}\delta_{r-j,\text{ev}}L(\chi_4,k-r+j)\til{h}_{r-j}(\tau)\Biggr)
\\
&=\frac{4}{(2\pi i)^{k}}\sum_{j=1}^{k-2}\Biggl(\frac{(-1)^{r-1}}{2^{j}}\binom{k-1-j}{r-1}\delta_{j,\text{od}}(1-\delta_{j,k-r})L(\chi_0,k-j)h_{j}(\tau)\\
&\quad\quad +\frac{(-1)^{r-j}}{2^{j}}\binom{k-j-1}{k-r-1}\delta_{j,\text{ev}}L(\chi_4,k-j)\til{h}_{j}(\tau)\Biggr).
\end{split}
\end{equation}
Here,
$$\delta_{k,\text{ev}}\ceq\left\{\begin{array}{cc}1 & (k:\text{even})\\ 0 & (k:\text{odd}) \end{array}\right.\ ,\quad 
\delta_{k,\text{od}}\ceq\left\{\begin{array}{cc}0 & (k:\text{even})\\ 1 & (k:\text{odd}) \end{array}\right. ,$$
and note that since $k-1$ is an odd number, the sum extends only up to $k-2$. We define
$$\mathcal{X}_j\ceq -\frac{4}{(2\pi i)^{k}}\frac{1}{2^j}L(\chi_0,k-j)h_{j}(\tau)\ ,\quad \mathcal{Y}_j\ceq \frac{4}{(2\pi i)^{k}}\frac{(-1)^j}{2^j}L(\chi_4,k-j)\til{h}_{j}(\tau).$$
For a prime number $p$, we have
$$\Im_{p}(h_{j})=\delta_{j,\text{od}}\frac{(-2\pi i)^j}{(j-1)!}\frac{1+\chi_0(p)p^{j-1}}{2^{j-1}}\ ,\quad \Im_{p}(\til{h}_{j})=i\cdot \delta_{j,\text{ev}}\frac{(-2\pi i)^j}{(j-1)!}\frac{1+\chi_4(p)p^{j-1}}{2^{j-1}}.$$
Thus, $\mathcal{X}_1,\mathcal{Y}_2,\dots,\mathcal{X}_{k-3},\mathcal{Y}_{k-2}$ are linearly independent over $\mathbb{C}$ (especially over $\mathbb{Q}$) because the Vandermonde determinant
$$\begin{vmatrix}
1 & 1 & 1^2 & \cdots & 1^{k-3}
\\
1 & 5 & 5^2 & \cdots & 5^{k-3}
\\
1 & 13 & 13^2 & \cdots & 13^{k-3}
\\
\vdots & \vdots & \vdots & \ddots & \vdots
\\
1 & p_{k-3} & p_{k-3}^2 & \cdots & p_{k-3}^{k-3}
\end{vmatrix}\quad \left(\begin{matrix}p_j\text{ is the }j\text{-th prime number}\\\text{such that it is congruent to 1 modulo 4}\end{matrix}\right)$$
 is nonzero. And, by \cref{eq:ImH_even}, we have
\begin{equation*}
\begin{pmatrix}\Im(\til{H}_{1,k-1})& \Im(\til{H}_{2,k-2}) & \dots &\Im(\til{H}_{k-1,1})\end{pmatrix}=
\begin{pmatrix}\mathcal{X}_1& \mathcal{Y}_2 & \dots &\mathcal{Y}_{k-2}\end{pmatrix} M_k,
\end{equation*}
where
\begin{equation}\label{eq:M_k-even}
M_k\ceq\left((-1)^r(1-\delta_{j,k-r})\delta_{j,\text{od}}\binom{k-j-1}{r-1}+(-1)^r\delta_{j,\text{ev}}\binom{k-j-1}{k-r-1}\right)_{\substack{1\leq j\leq k-2\\1\leq r\leq k-1}}.
\end{equation}
Therefore $\dim_{\mathbb{Q}}\Im\DHk=\rank M_{k}$.

We assume $k$ is an odd integer. By an argument similar to the case when $k$ is even, we obtain
\begin{equation}\label{eq:odd-IMH}
\begin{split}
\Im(\til{H}_{r,k-r})
&=\frac{4}{(2\pi i)^{k}}\sum_{j=1}^{k-2}\Biggl(\frac{(-1)^{r-1}}{2^{j}}\binom{k-j-1}{r-1}L(\chi_0,k-j) (1-\delta_{j,k-r})\delta_{j,\text{ev}}h_{j}(\tau) \\
&\quad\quad + \frac{(-1)^{r-j}}{2^{j}}\binom{k-j-1}{k-r-1} L(\chi_4,k-j) \delta_{j,\text{od}}\til{h}_{j}(\tau)\Biggr),
\end{split}
\end{equation}
and $\dim_{\mathbb{Q}}\Im\DHk=\rank M_{k}$, where
\begin{equation}\label{eq:M_k-odd}
M_k\ceq \left((-1)^r (1-\delta_{j,k-r})\delta_{j,\text{ev}}\binom{k-j-1}{r-1} + (-1)^r \delta_{j,\text{od}}\binom{k-j-1}{k-r-1}
\right)_{\substack{1\leq j\leq k-2\\1\leq r\leq k-1}}.
\end{equation}
Moreover, by \cref{lem:H=0} below, we also obtain
\begin{equation}\label{eq:odd-cond}
\dim_{\mathbb{Q}}\DHk\geq \dim_{\mathbb{Q}}\Im\DHk+1=\rank M_{k}+1.
\end{equation}

\begin{lem}\label{lem:H=0} For any odd integer $k\geq 3$, we have
$$\Im\left(\sum_{r=1}^{k-1}2^{r-2}\til{H}_{k-r,r}+\frac{1}{2}\til{H}_{k-1,1}\right)=0
\quad \text{and}\quad 
\sum_{r=1}^{k-1}2^{r-2}\til{H}_{k-r,r}+\frac{1}{2}\til{H}_{k-1,1}\neq 0.$$
\end{lem}

\begin{proof}
By using \cref{eq:odd-IMH}, we can check the first equation. We assume that the left-hand side in the second equation is equal to 0, thus, $H_k^0(X,2X)+H_k^0(X,0)=0$. Upon comparing \cref{eq:HtoG}, we obtain
$$-\frac{1}{2(k-1)}\left(H_k^0(X,0)+\wh{H}_k^0(X,X)\right)=\til{G}^0_kX^{k-2}.$$
Similar to the argument in proof of \cref{H+H=0}, we have
 $$\frac{\chi_4(n)}{2}E_{k-2}\left(\frac{n+1}{2}\right)
 =\chi_0(n)\sum_{w=1}^{n-1}\chi_4(w+1)\left(\frac{w}{2}\right)^{k-2}+\frac{\chi_4(n+1)}{2}\left(\frac{n}{2}\right)^{k-1}\quad (\forall n\geq 1),$$
 and this equation does not hold when $n$ is even. Hence, we have a contradiction.
\end{proof}

Let $k\geq 3$ be an integer and $V_k$ be the set homogeneous rational polynomials of degree $k-2$ in two variables. Then, the action of $\Gamma=GL_2(\mathbb{Z})$ on $V_k$ is defined by
$$(P|_{\gamma})(X,Y)=P(aX+bY,cX+dY)\quad \left(P(X,Y)\in V_k,\ \gamma=\begin{pmatrix}a&b\\c&d\end{pmatrix}\in \Gamma\right).$$
Furthermore, we extend the action of $\Gamma$ on $V_k$ to an action of the group ring $\mathbb{Q}[\Gamma]$ by linearity. Note that this action varies based on the parity of $k$, i.e., for $J\ceq\begin{pmatrix}-1&0\\0&-1\end{pmatrix}$, $P|_{J}=P$ if $k$ is an even integer, and $P|_{J}=-P$ if $k$ is an odd integer. We define
$$
\delta\ceq\begin{pmatrix}
-1 & 0 \\
0 & 1 
\end{pmatrix}\ ,\quad
\ep\ceq\begin{pmatrix}
0 & 1 \\
1 & 0 
\end{pmatrix}\ ,\quad
U\ceq\begin{pmatrix}
1 & -1 \\
1 & 0 
\end{pmatrix}\ ,\quad
T\ceq\begin{pmatrix}
1 & 1 \\
0 & 1 
\end{pmatrix},
$$
and then we can check easily
$$
\ep U^2=JU\ep\ ,\quad U^3=J\ ,\quad U\ep\delta=T.
$$
The $\mathbb{Q}$-linear map from $V_k$ to $V_k$ defined by $\gamma\in\mathbb{Q}[\Gamma]$ is also denoted by $\gamma$, and the image under $\gamma$ of a subspace $V\subset V_k$ is denoted by $V|_{\gamma}$. Especially, we denote $V^{\text{ev}}\ceq V|_{\frac{1+\delta}{2}}$ and $V^{\text{od}}\ceq V|_{\frac{1-\delta}{2}}$. We can check easily that $P^{\text{ev}}(X,Y)\ceq P|_{\frac{1+\delta}{2}}(X,Y)$ (resp. $P^{\text{od}}(X,Y)\ceq P|_{\frac{1-\delta}{2}}(X,Y)$) consists of terms in $P(X,Y)$ whose degree in $X$ is even (resp. odd), and $V_k=V_k^{\text{ev}}\oplus V_k^{\text{od}}$. Let $\Delta_k\in\mathbb{Q}[\Gamma]$ be
\begin{equation*}
\Delta_k:=\ep \left(\frac{1+\delta}{2}-(-1)^{k}\frac{1-\delta}{2}U-\frac{1+\delta}{2}U\ep\right).
\end{equation*}

\begin{prop}\label{prop:M_k-Delta}
For any integer $k\geq 3$,
$$\rank M_{k}=\dim\Img\Delta_k.$$
\end{prop}

\begin{proof}
Let $P(X,Y)=\displaystyle{\sum_{n=0}^{k-2}a_nX^{n}Y^{k-2-n}}$, then we have
\begin{align*}
(P|_{\frac{1+\delta}{2}})(X,Y)
&=\sum_{i=0}^{k-2}\delta_{i,\text{ev}}a_{i}X^{i}Y^{k-2-i},
\\
(P|_{\frac{1-\delta}{2}U})(X,Y)
&=(-1)^{k}\sum_{i=0}^{k-2}\sum_{n=0}^{k-2}(-1)^{i}\delta_{n,\text{od}}a_n\binom{n}{k-2-i}X^iY^{k-2-i},
\\
(P|_{\frac{1+\delta}{2}U\ep})(X,Y)
&=\sum_{i=0}^{k-2}\sum_{n=0}^{k-2}(-1)^{i}\delta_{n,\text{ev}}a_n\binom{n}{i}X^iY^{k-2-i},
\end{align*}
and as a result, we have
\begin{equation*}
\begin{split}
(P|_{\Delta_k})(X,Y)
&=\sum_{i=0}^{k-2}\Biggl(
-\sum_{n=0}^{k-2}(-1)^{i}\delta_{k-2-n,\text{od}}a_n\binom{k-2-n}{k-2-i}
\\
&\quad \quad -\sum_{n=0}^{k-2}(-1)^{i}\delta_{k-2-n,\text{ev}}(1-\delta_{k-2-n,i})a_n\binom{k-2-n}{i}
\Biggr)X^iY^{k-2-i}.
\end{split}
\end{equation*}
Thus, if $k$ is an even integer,
\begin{equation*}
\begin{split}
(P|_{\Delta_k})(X,Y)
=\sum_{i=1}^{k-1}
\sum_{n=1}^{k-1}\Biggl((-1)^{i}\delta_{n,\text{od}}(1-\delta_{n,k-i})\binom{k-1-n}{i-1}
+(-1)^{i}\delta_{n,\text{ev}}\binom{k-1-n}{k-1-i}\Biggr)a_{n-1}
X^{i-1}Y^{k-1-i},
\end{split}
\end{equation*}
and if $k$ is an odd integer,
\begin{equation*}
\begin{split}
(P|_{\Delta_k})(X,Y)
=\sum_{i=1}^{k-1}
\sum_{n=1}^{k-1}\Biggl((-1)^{i}\delta_{n,\text{ev}}(1-\delta_{n,k-i})\binom{k-1-n}{i-1}
+(-1)^{i}\delta_{n,\text{od}}\binom{k-1-n}{k-1-i}\Biggr)a_{n-1}
X^{i-1}Y^{k-1-i}.
\end{split}
\end{equation*}
Therefore, upon comparing the above with the definitions \cref{eq:M_k-even,eq:M_k-odd}, it can be observed that $M_k$ corresponds to the representation matrix of $\Delta_k$ by excluding $n=k-1$-th row, where only zeros appear. Hence, $\rank M_k = \dim_{\mathbb{Q}} \Img \Delta_k$.
\end{proof}

\begin{prop}\label{prop:M-even-dimension}
For any even integer $k\geq 2$,
\begin{equation*}
\dim_{\mathbb{Q}}\Im\DHk=\rank M_{k}=\left[\frac{3k}{4}\right]-1.
\end{equation*}
\end{prop}

\begin{proof} By \cref{prop:M_k-Delta}, we only need to consider $\dim\Img\Delta_k$. From
\begin{align*}
(1+\delta)\left(\frac{1+\delta}{2}-\frac{1-\delta}{2}U-\frac{1+\delta}{2}U\ep\right)
&=(1+\delta)(1-U\ep),
\\
(1-\delta)\left(\frac{1+\delta}{2}-\frac{1-\delta}{2}U-\frac{1+\delta}{2}U\ep\right)
&= -(1-\delta)U,
\end{align*}
we have
$$\Img\ep\Delta_k=V_k^{\text{ev}}|_{(1-U\ep)} + V_k^{\text{od}}|_{U},$$
and by multiplying both sides by the invertible element $U^2$ from the right,
$$\Img\ep\Delta_k|_{U^2}=V_k^{\text{ev}}|_{U^2(1-\ep)} + V_k^{\text{od}}.$$
For $P(X,Y)=\displaystyle{\sum_{\substack{r=0\\r:\text{even}}}^{k-2}a_rX^{k-2-r}Y^{r}}\in V^{\text{ev}}$, 
\begin{equation*}
(P|_{U^2(1-\ep)})^{\text{ev}}(X,Y)
=\sum_{\substack{i=0\\i:\text{even}}}^{k-2}\sum_{\substack{r=0\\r:\text{even}}}^{k-2}(-1)^{i}a_r\left(\binom{r}{i}
-\binom{r}{k-2-i}\right)X^{i}Y^{k-2-i}.
\end{equation*}
Thus, the dimension of $(V_k^{\text{ev}}|_{U^2(1-\ep)})^{\text{ev}}$ equales
\begin{alignat*}{3}
\rank\left(\binom{r}{i}-\binom{r}{k-2-i}\right)_{\substack{0\leq i,r\leq k-2\\i,r:\text{even}}}
&\overset{(\text{i})}{=}\rank\left(\binom{2r}{2i}-\binom{2r}{k-2-2i}\right)_{\substack{0\leq i\leq\left[\frac{k}{4}\right] -1\\0\leq r\leq \frac{k-2}{2}}}
\\
&\overset{(\text{ii})}{=}\rank\left(\binom{2r}{2i}\right)_{0\leq i,r\leq\left[\frac{k}{4}\right]-1}
&=\left[\frac{k}{4}\right].
\end{alignat*}
Here, the equality in (i) is obtained from the symmetry with respect to $i$ of the matrix on the left-hand side, and the equality in (ii) is obtained by restricting the range of $r$. Therefore
$$\dim\Img\Delta_k=\dim (V_k^{\text{ev}}|_{U^2(1-\ep)})^{\text{ev}}+\dim V_k^{\text{od}}=\left[\frac{k}{4}\right]+\frac{k-2}{2}=\left[\frac{3k}{4}\right]-1.$$
\end{proof}

\begin{prop}\label{prop:M-odd-dimension}
For any odd integer $k\geq 3$,
\begin{equation*}
\dim_{\mathbb{Q}}\Im\DHk=\rank M_{k}=k-2.
\end{equation*}
\end{prop}

\begin{proof} By \cref{prop:M_k-Delta}, we only need to consider $\dim\Img\Delta_k$. In a similar way as in the case when $k$ is even, we have
$$\Img \ep\Delta_k = V_k^{\text{ev}}|_{(1-U\ep)} + V_k^{\text{od}}|_{U}.$$
By multiplying both sides by the invertible element $\delta$ from the right, we obtain
$$\Img \ep\Delta_k |_{\delta} = V_k^{\text{ev}}|_{(1-T)}+V_k^{\text{od}}|_{U\delta}.$$
Then we know $\dim V_k^{\text{ev}}|_{(1-T)}=\dim V_k^{\text{ev}}-1$ by $\Ker (1-T)=\{a Y^{k-2}\mid a\in\mathbb{Q}\}\cong \mathbb{Q}$.
And also, by multiplying both sides by the invertible element $U^2$ from the right,
$$\Img \ep\Delta_k |_{U^2} = V_k^{\text{ev}}|_{U^2(1+\ep)} + V_k^{\text{od}}.$$
For $0\neq P(X,Y)\in V_k^{\text{od}}$, $P(X,Y)|_{\ep}\notin V_k^{\text{od}}$ since $k$ is odd. On the other hand, $P(X,Y)\in V_k^{\text{ev}}|_{U^2(1+\ep)}$ satisfy $P(X,Y)|_{\ep}=P(X,Y)$. Therefore, we obtain $\Img \ep\Delta_k |_{U^2} = V_k^{\text{ev}}|_{U^2(1+\ep)} \oplus V_k^{\text{od}}$, and as a result,
$$\dim\Img \Delta_k = \dim\Img \ep\Delta_k |_{U^2} = \dim V_k^{\text{ev}} -1 + \dim V_k^{\text{od}} = k-2$$
since $\dim V_k^{\text{ev}}|_{U^2(1+\ep)}=\dim V_k^{\text{ev}}|_{(1-T)}$.
\end{proof}

\section{Proof of \cref{thm:dimension}}\label{sec:proof2}

If $k\geq 3$ is an odd integer, we can prove \cref{eq:all-dimension} because we have \cref{eq:odd-cond,prop:M-odd-dimension}.

In the following, we assume $k$ is even. The set $B_{M\DHk}$ given by \cref{H-base} is contained within $M\DHk$ which is a subspace of $\Ker\Im$ by \cref{thm:shuffle}. Furthermore, we have
$$\dim_{\mathbb{Q}}\ker\Im = \dim_{\mathbb{Q}}\DHk-\dim_{\mathbb{Q}}\Im\DHk \leq k-1-\left(\left[\frac{3k}{4}\right]-1\right)=\left[\frac{k-2}{4}\right]+1=|B_{M\DHk}|$$
by \cref{prop:M-even-dimension}. Therefore, we can prove \cref{eq:all-dimension} and that $B_{M\DHk}$ is a basis of $M\DHk$ if we prove that the elements of $B_{M\DHk}$ are linearly independent over $\mathbb{Q}$.

Since $\til{H}_k$ and $\til{G}_k$ have non-zero constant terms of Fourier expansions only at the cusp $[i\infty]$ among $\Gamma_0(4)\setminus\mathbb{P}^1(\mathbb{Q})=\{[i\infty],[0],[1/2]\}$, there exist $\lambda_{k,r}\in\mathbb{Q}$ such that $\til{H}_{r}\til{H}_{k-r}-\lambda_{k,r}\til{G}_k\in S\DHk$. Denote by
$$B_{S\DHk}\ceq\{\til{H}_{r}\til{H}_{k-r}-\lambda_{k,r}\til{G}_k\mid 3\leq r\leq k-3:\text{odd}\}\subset S\DHk.$$
Therefore, all claims of \cref{thm:dimension} are completely proven if the claim
\begin{equation*}
(\spadesuit) \text{ The elements of }B_{S\DHk} \text{ are linearly independent over }\mathbb{C}
\end{equation*}
is true by the fact that Eisenstein series and any cusp form are orthogonal. If $k=2$ and $4$, \cref{thm:dimension} holds trivially. Thus, in the following, we assume that $k$ is greater than $4$.
 
The following proof builds upon the arguments by Fukuhara and Yang in \cite{Fuk-Yan} and \cite{Fuk-Yan-2}. Futhermore, it is primarily an adjustment of the arguments presented by Antoniadis in \cite{JA}, tailored to a specific case. See also \cite{Koh-Zag} by Kohnen and Zagier.

\subsection{Transformation of the claim $\spadesuit$}
For $f,g\in M_k(\Gamma_0(4))$ where at least one of them belongs to $S_k(\Gamma_0(4))$, we define the Petersson inner product by
$$(f,g):=\int_{\Gamma_{0}(4)\setminus\mathbb{H}}f(\tau)\overline{g(\tau)}y^{k-2}\ d\tau\ ,\quad (\tau=x+iy).$$
First, we calculate the Petersson inner product of $\til{H}_{r}\til{H}_{s}$ and a cusp form by using the ``Rankin-Selberg'' method.

\begin{prop}\label{prop:Rankin-Selberg-method}
Let $s$ be an odd integer such that $3\leq s\leq k-3$, and set $r=k-s$. For any $\displaystyle{f(\tau)=\sum_{n=1}^{\infty}a_nq^n} \in S_{k}(\Gamma_0(4))$, we have 
\begin{align*}
(\til{H}_{r}\til{H}_{s},f)
&= \rho_{r,s}(0)L(\chi_4,r)\sum_{n=1}^{\infty}\frac{\overline{a}_n\til{\sigma}_{s-1}(n)}{n^{r+s-1}}\quad \left(\til{\sigma}_k(n)\ceq\sum_{d|n}\chi_4(d)d^{k}\right)
\end{align*}
Especially, if we assume that $f$ is normalized Hecke eigenform of level 1, 2, and 4, we have
\begin{equation}\label{eq:petersson-result}
(\til{H}_{r}\til{H}_{s},f) = \rho_{r,s}(0)L(f,k-1)L(f_{\chi_4},r).
\end{equation}
Here, $\displaystyle{L(f,k)\ceq\sum_{n=1}^{\infty}\frac{a_n}{n^k}}$ and $\displaystyle{f_{\chi_4}(\tau)\ceq\sum_{n=1}^{\infty}a_n\chi_4(n)q^n}$ for $f(\tau)$, and
\begin{equation*}
\rho_{r,s}(t)\ceq\frac{-2i}{4^{s-1}}\frac{1}{(s-1)!}\frac{1}{(2\pi i)^{r}}\frac{\Gamma(k-1+t)}{(4\pi)^{k-1+t}}\quad(t\in\mathbb{C}).
\end{equation*}
\end{prop}

\begin{proof}
We consider the holomorphic function
$$\til{H}_{r}^{(t)}(\tau)\ceq \til{H}_{r}(\tau,t)
\ceq \frac{2}{(2\pi i)^r}\sum_{0\prec 4m\tau+n}\frac{y^{t}}{|4m\tau+n|^{2t}}\frac{\chi_4(n)}{(4m\tau+n )^r}.$$
for $t$ with $\Real(2t+r)>2$. This function satisfies $\til{H}_{r}^{(0)}(\tau)=\til{H}_{r}(\tau)$. Using the modularity of $\til{H}_{s}(\tau)$, we have
\begin{align*}
\til{H}_{r}(\tau,t)\til{H}_{s}(\tau)
&=\frac{2}{(2\pi i)^r}\sum_{0\prec 4m\tau+n}\frac{y^{t}}{|4m\tau+n|^{2t}}\frac{\chi_4(n)}{(4m\tau+n )^r}\til{H}_{s}(\tau)
\\
&=\frac{2}{(2\pi i)^r}\sum_{\substack{\text{gcd}(4m,n)=1\\0\prec 4m\tau+n}}\sum_{u=1}^{\infty}\frac{\chi_4(u)}{u^{r+2t}}\frac{y^{t}}{|4m\tau+n|^{2t}}\frac{\chi_4(n)}{(4m\tau+n )^r}\til{H}_{s}(\tau)
\\
&=\frac{2L(\chi_4,r+2t)}{(2\pi i)^r}\sum_{\gamma\in \Gamma_{\infty}\setminus\Gamma_0(4)}\left(\Img(\tau)^{t}\til{H}_{s}(\tau)\right)|_{r+s}\gamma
\quad
\left(\Gamma_{\infty}=\left\{\pm\begin{pmatrix}1&b\\0&1\end{pmatrix}\middle| b\in\mathbb{Z}\right\}\right).
\end{align*}
Denoting by $\{h_n\}_{n\geq 0}$ the Fourier coefficients of $\til{H}_{s}(\tau)$ and using the ``Rankin unfolding trick'', we have
\begin{align*}
(\til{H}_{r}^{(t)}\til{H}_{s},f)
&=\int_{\Gamma_{0}(4)\setminus\mathbb{H}}\overline{f(\tau)}\til{H}_{r}^{(t)}\til{H}_{s} y^{k-2}dxdy
\\
&=2\frac{L(\chi_4,r+2t)}{(2\pi i)^r}\int_{\Gamma_{0}(4)\setminus\mathbb{H}}\sum_{\gamma\in \Gamma_{\infty}\setminus\Gamma_0(4)}\overline{f(\tau)}\left(\Img(\tau)^{t}\til{H}_{s}(\tau)\right)|_{r+s}\gamma \ y^{k-2}dxdy
\\
&=2\frac{L(\chi_4,r+2t)}{(2\pi i)^r}\int_{\Gamma_{\infty}\setminus\mathbb{H}}\overline{f(\tau)}\til{H}_{s}(\tau) y^{k-2+t}dxdy
\notag
\\
&=2\frac{L(\chi_4,r+2t)}{(2\pi i)^r}\int_{0}^{\infty}\int_{0}^{1}\left(\sum_{n=1}^{\infty} \overline{a}_n e^{2\pi i n(-x+iy)}\right)\left(\sum_{n=0}^{\infty}h_ne^{2\pi i n(x+iy)}\right) y^{k-2+t}dxdy
\\
&=2\frac{L(\chi_4,r+2t)}{(2\pi i)^r}\int_{0}^{\infty}\sum_{n=1}^{\infty}\overline{a}_n h_ne^{-4\pi ny} y^{k-2+t}dy
\\
&=2\frac{L(\chi_4,r+2t)}{(2\pi i)^r}\frac{\Gamma(k-1+t)}{(4\pi)^{k-1+t}}\sum_{n=1}^{\infty}\frac{\overline{a}_n h_n}{n^{k-1+t}}
\end{align*}
for $s < \frac{k-1}{2}+\Real(t)$ because the series of the last equation converges absolutely, with $a_n=O(n^{\frac{k}{2}-\frac{1}{2}+\ep})$ for $\ep>0$ (by Deligne's theorem) and $h_n=O(n^{s-1})$. Thus, we obtain
\begin{equation*}
(\til{H}_{r}^{(t)}\til{H}_{s},f)
= \rho_{r,s}(t)L(\chi_4,r+2t)\sum_{n=1}^{\infty}\frac{\overline{a}_n \til{\sigma}_{s-1}(n)}{n^{k-1+t}}\quad \left(s < \frac{k-1}{2}+\Real(t)\right)
\end{equation*}
from
\begin{equation*}
h_n=\frac{-i}{4^{s-1}}\frac{\til{\sigma}_{s-1}(n)}{(s-1)!}\quad (n>0).
\end{equation*}
If we assume that $f\in S_{k}(\Gamma_0(4))$ is a normalized Hecke eigenform, the property
\begin{equation}\label{heck-coefficient}
a_na_m=\sum_{d|n,m}\chi_0(d)d^{k-1}a_{\frac{mn}{d^2}}
\end{equation}
leads to
\begin{align*}
L(\chi_4,r+2t)\sum_{n=1}^{\infty}\frac{a_{n}\til{\sigma}_{s-1}(n)}{n^{r+s-1+t}}&=\left(\sum_{p=1}^{\infty}\frac{\chi_0(p)\chi_4(p)}{p^{r+2t}}\right)\sum_{n=1}^{\infty}\left(\sum_{m=1}^{\infty}\frac{a_{nm}n^{s-1}\chi_4(n)}{(nm)^{r+s-1+t}}\right)
\\
&=\sum_{m=1}^{\infty}\sum_{n=1}^{\infty}\sum_{d|n}a_{\frac{n}{d}m}\chi_0(d)\chi_4(n)\frac{1}{d^{t}}\frac{1}{n^{r+t}}\frac{1}{m^{r+s-1+t}}
\\
&=\sum_{n=1}^{\infty}\sum_{m=1}^{\infty}\sum_{d|n,m}\chi_0(d)d^{k-1}a_{\frac{nm}{d^2}}\chi_4(n)n^{s-1}\frac{1}{(nm)^{k-1+t}}
\\
&=\sum_{n=1}^{\infty}\sum_{m=1}^{\infty}a_na_m\chi_4(n)n^{s-1}\frac{1}{(nm)^{k-1+t}}
\quad (\because\cref{heck-coefficient})
\\
&=\left(\sum_{n=1}^{\infty}\frac{a_n}{n^{k-1+t}}\right)
\left(\sum_{n=1}^{\infty}\frac{a_n\chi_4(n)}{n^{r+t}}\right).
\end{align*}
Therefore, by $a_n\in\mathbb{R}$ for all $n$, we obtain 
\begin{equation}\label{Rankin-Selberg-t}
(\til{H}_{r}^{(t)}\til{H}_{s},f) = \rho_{r,s}(t)L(f,k-1+t)L(f_{\chi_4},r+t)
\quad \left(s < \frac{k-1}{2}+\Real(t)\right).
\end{equation}
By similar calculations, \cref{Rankin-Selberg-t} is obtained for normalized Hecke eigenforms of level 1 and level 2 as well. Here, $(\til{H}_{r}^{(t)}\til{H}_{s},f)$ is defined for $\Real(2t+r)>2$ and becomes a holomorphic function for $t$ because
$$(\til{H}_{r}^{(t)}|_r\gamma)(\tau)=O(y^t)+O(y^{1-r-t})\quad\text{as}\quad\tau\to i\infty$$
for all $\gamma\in SL_2(\mathbb{Z})$ (see \cite[A3.5.]{Shimura}). Hence, \cref{Rankin-Selberg-t} folds for $\Real(2t+r)>2$ by the identity theorem. Therefore, by setting $t=0$, we obtain $(\til{H}_{r}\til{H}_{s},f) = \rho_{r,s}(0)L(f,k-1)L(f_{\chi_4},r)$.
\end{proof}

Let
\footnotetext[3]{Following \cite{Dia-Shu}, we simply refer to a normalized Hecke eigenform in the space of newfroms as a newform.}
\begin{align*}
B_{k}&\ceq\{f(n\tau):f(\tau) \text{ is a newform\textcolor{orange}{\footnotemark[3]} of level }M \text{ and } nM|4\},
\\
\til{B}_{k}&\ceq\{f(\tau):f(\tau) \text{ is a newform of level }M \text{ and } M|4\}.
\end{align*}
$\til{B}_k$ is a subset of $B_{k}$, and the cardinalities of $B_{k}$ and $\til{B}_{k}$ are
\begin{align*}
d\ceq(k-4)/2 \quad \text{and} \quad
\til{d}\ceq\dim S_k(\Gamma_0(4))-\dim S_k(\Gamma_0(2))=[(k-2)/4],
\end{align*}
respectively. We label the functions of $B_k$ by $f_1,\dots,f_d$ such that $f_1,\dots,f_{\til{d}}\in\til{B}_k$. Note that the cardinality of $B_{S\DHk}$ is also $\til{d}$. If we define $g(\tau)=f(2\tau)$ for $f(\tau)\in S_{k}(\Gamma_0(4))$, $\til{\sigma}_k(2n)=\til{\sigma}_k(n)$ leads to
\begin{equation*}
(\til{H}_{r}^{(t)}\til{H}_s,g)
=\rho_{r,s}(t)L(\chi_4,r+t)\sum_{n=1}^{\infty}\frac{a_{n}\til{\sigma}_{s-1}(2n)}{(2n)^{k-1+t}}
=\frac{1}{2^{k-1+t}}(\til{H}_{r}^{(t)}\til{H}_s,f).
\end{equation*}
Thus, 
$$(\til{H}_{r}\til{H}_s,g)=\frac{1}{2^{k-1}}(\til{H}_{r}\til{H}_s,f).$$
Similarly, if define $g(\tau)=f(4\tau)$, then
$$(\til{H}_{r}\til{H}_s,g)=\frac{1}{4^{k-1}}(\til{H}_{r}\til{H}_s,f).$$
Therefore, since $B_k$ is a basis of $S_k(\Gamma_0(4))$, $\spadesuit$ is equivalent to
$$\rank\Bigl((\til{H}_{k-2j-1}\til{H}_{2j+1}-\lambda_{k,2j+1}\til{G}_k,f_i)\Bigr)_{\substack{1\leq i\leq d\\1\leq j\leq \til{d}}}=
\rank\Bigl((\til{H}_{k-2j-1}\til{H}_{2j+1},f_i)\Bigr)_{1\leq i,j\leq \til{d}}=\til{d}.$$
Furthermore, since $\rho_{r,s}(0)\neq 0$ and $L(f,k-1)\neq 0$ in \cref{eq:petersson-result}, $\spadesuit$ is equivalent to
\begin{equation}\label{eq:independence2}
\rank\Bigl(L(f_{i\chi_4},2j+1)\Bigr)_{1\leq i,j\leq \til{d}} = \til{d}.
\end{equation}
In the following, let $w=k-2$. We define $n$-period to prove \cref{eq:independence2}. For an integer $n$ with $0\leq n\leq w$ and $f(\tau)\in S_k(\Gamma_0(N))$, the $n$-period of $f$ is
\begin{equation*}
r_{n}(f)\ceq\int_{0}^{i\infty}f(\tau)\tau^n\ d\tau.
\end{equation*}
For $f(\tau)\in S_k(\Gamma_0(N))$, by Mellin transformation, we have
\begin{equation*}
r_n(f_{\chi_4})
=\frac{n!}{(-2\pi i)^{n+1}}L(f_{\chi_4},n+1).
\end{equation*}
Thus, $\spadesuit$ is equivalent to
\begin{equation*}
\rank\Bigl(r_{2j}(f_{i\chi_4})\Bigr)_{1\leq i,j\leq \til{d}} = \til{d}.
\end{equation*}
Since the map $S_k(\Gamma_0(4))\ni f\mapsto r_n(f_{\chi_4})\in\mathbb{C}$ is linear, there exists a unique cusp form $R_{w,n}\in S_k(\Gamma_0(4))$ such that
\begin{equation*}
r_n(f_{\chi_4})=(2i)^{w+1}(f,R_{w,n}).
\end{equation*}
Thus, since we have $f_{\chi_4}(\tau)=0$ if $f(\tau)=g(2\tau)$, $\spadesuit$ is equivalent to
\begin{equation*}
\rank\Bigl((f_i,R_{w,2j})\Bigr)_{\substack{1\leq i\leq d\\1\leq j\leq \til{d}}} = \til{d}.
\end{equation*}
Moreover, $\spadesuit$ is equivalent to the independence of elements of $\{R_{w,2j}\mid 1\leq j\leq \til{d}\}$, because $B_k$ is a basis of $S_k(\Gamma_0(4))$. Therefore, to prove $\spadesuit$, it is enough to prove that the determinant of the matrix
\begin{equation*}
A_{w}\ceq \Bigl((2i)^{w+1}(R_{w,2i-1},R_{w,2j})\Bigr)_{1\leq i,j\leq \til{d}}
=\Bigl(r_{2j}(R_{w,2i-1,\chi_4})\Bigr)_{1\leq i,j\leq \til{d}}
\end{equation*}
is non-zero. We prove $\det A_w\neq 0$ in the next section.

\subsection{The explicit representation of $A_w$ and completion of the proof}

We confirm the fact (c.f. \cite[\S 3 Proposition17]{Kob}): Let $f(\tau)\in M_k(\Gamma_0(M))$ and $\chi$ be a primitive Dirichlet character modulo $N$, then
\begin{equation*}
f_{\chi}(\tau)=\frac{G(\chi)}{N}\sum_{v=0}^{N-1}\overline{\chi}(v)f(\tau-v/N),
\end{equation*}
where $G(\chi)=\displaystyle{\sum_{j=0}^{N-1}\chi(j)e^{2\pi i j/N}}$ is the Gauss sum. Especially, for $\displaystyle{\eta_{e}:=\begin{pmatrix}1&(-1)^{e}/4\\0&1\end{pmatrix}}$,
\begin{equation}\label{eq:f-chi4}
f_{\chi_4}(\tau)=\frac{1}{2i}f|_{k}(\eta_{0}-\eta_{1})(\tau).
\end{equation}
Here, $f|_k\gamma(\tau)=(\det\gamma)^{k/2}(c\tau+d)^{-k}f\left(\frac{a\tau+b}{c\tau+d}\right)$ for $\displaystyle{\gamma:=\begin{pmatrix}a&b\\ c&d\end{pmatrix}}\in GL_2^{+}(\mathbb{Q})$.

In the following, let $\til{n}=w-n$ for $0\leq n\leq w=k-2$. 

\begin{lem}\label{lem:RR-explicit} For an integer $n$ with $0<n<w$, set
\begin{align*}
\alpha_{w,n}:=(-1)^{n}4\pi \binom{w}{n},
\end{align*}
then for any even integer $k\geq 6$,
\begin{equation*}
R_{w,n}(\tau)=\alpha_{w,n}^{-1}\sum_{\gamma\in\Gamma_0(4)}
\left.\frac{1}{\tau^{\til{n}+1}}\right|_{k}(\eta_{0}-\eta_{1})\gamma.
\end{equation*}
\end{lem}

\begin{proof} 
We consider
\begin{align*}
P_{w,n}(\tau)
&\ceq\sum_{\gamma\in\Gamma_0(4)}
\left.\frac{1}{\tau^{\til{n}+1}}\right|_{k}(\eta_{0}-\eta_{1})\gamma .
\end{align*}
Then, by Cauchy's theorem and $\til{n}$-fold integration by parts, we have
\begin{align*}
(f,P_{w,n})
&=\int_{0}^{\infty}\int_{-\infty}^{\infty}\left(\frac{f(x+iy)}{(x-iy+1/4)^{\til{n}+1}}-\frac{f(x+iy)}{(x-iy-1/4)^{\til{n}+1}}\right) \ y^{w}dxdy
\\
&=\int_{0}^{\infty}\frac{2\pi i}{\til{n}!}\left(f^{(\til{n})}(2iy-1/4)-f^{(\til{n})}(2iy+1/4)\right) \ y^{w}dy
\\
&= \frac{2\pi i}{\til{n}!}\frac{w!}{(w-\til{n})!}\left(-\frac{1}{2i}\right)^{\til{n}}\int_{0}^{\infty} \left(f(2iy-1/4)-f(2iy+1/4)\right)y^{w-\til{n}}dy
\\
&= 2\pi i\binom{w}{n}\left(-\frac{1}{2i}\right)^{\til{n}-1}\int_{0}^{\infty} f_{\chi_4}(2iy) y^{n}dy
\\
&= (-1)^{n+1}2\pi i\binom{w}{n}\left(\frac{1}{2i}\right)^{w}r_n(f_{\chi_4}).
\end{align*}
Thus, we obtain this lemma.
\end{proof}

\begin{lem}\label{lem:atkin-lehner}
For any integer $n$ with $0\leq n\leq w$, we have $R_{w,n}=(-1)^{n}4^{\til{n}-n}R_{w,\til{n}}$ and also
\begin{equation*}
r_m(R_{w,n,\chi_4})=(-1)^{n+m}4^{\til{n}+\til{m}-n-m}r_{\til{m}}(R_{w,\til{n},\chi_4}).
\end{equation*}
\end{lem}

\begin{proof}
Because the Atkin-Lehner involution $W_4=\begin{pmatrix}0&-1\\4&0\end{pmatrix}$ is an adjoint operator of Petersson inner product, by $W_4\eta_{e}=\begin{pmatrix}0&-1\\4&(-1)^e\end{pmatrix}$ and \cref{eq:f-chi4}, we have
\begin{align*}
(f,R_{w,n}|_kW_4)
&=\left(f|_kW_4,R_{w,n}\right)
\\
&=\frac{1}{(2i)^{w+2}}\left(\int_0^{i\infty}\frac{2^{k}}{(4\tau+1)^k}f(\frac{-1}{4\tau+1})\tau^{n}d\tau
-\int_0^{i\infty}\frac{2^{k}}{(4\tau-1)^k}f(\frac{-1}{4\tau-1})\tau^{n}d\tau\right).
\end{align*}
Then making changes of variable $\frac{-1}{4\tau+1}+1\mapsto \tau$ and $\frac{-1}{4\tau-1}\mapsto \tau$ in the first and the second integral respectively, we obtain
\begin{align*}
(f,R_{w,n}|_kW_4)
&=\frac{1}{(2i)^{w+2}}\frac{2^{k}}{4^{n+1}}\int_{0}^{1}f(\tau)\left((-1)^n\tau^{n}(\tau-1)^{\til{n}} + \tau^{\til{n}}(\tau-1)^n \right) d\tau.
\end{align*}
Thus, we have
\begin{align*}
(f,R_{w,n}|_kW_4)
&=(-1)^{n}4^{\til{n}-n}\frac{1}{(2i)^{w+2}}\frac{2^{k}}{4^{\til{n}+1}}\int_{0}^{1}f(\tau)\left(\tau^{n}(\tau-1)^{\til{n}} + (-1)^{\til{n}}\tau^{\til{n}}(\tau-1)^n \right) d\tau
\\
&=(-1)^{n}4^{\til{n}-n}\left(f,R_{w,\til{n}}|_kW_4\right).
\end{align*}
Therefore, we obtain $R_{w,n}=(-1)^{n}4^{\til{n}-n}R_{w,\til{n}}$. Furthermore, we have
\begin{align*}
r_{m}(R_{w,n,\chi_4})
&=(2i)^{w+1}(R_{w,n},R_{w,m})
\\
&=(-1)^{n+m}4^{\til{n}+\til{m}-n-m}(2i)^{w+1}(R_{w,\til{n}},R_{w,\til{m}})
=(-1)^{n+m}4^{\til{n}+\til{m}-n-m}r_{\til{m}}(R_{w,\til{n},\chi_4}).
\end{align*}
\end{proof}

For $\gamma=\begin{pmatrix}a&b\\c&d\end{pmatrix}\in SL_2(\mathbb{Z})$, we define
\begin{align*}
t_1^{e_1,e_2}(\gamma)&\ceq c,
\\
t_2^{e_1,e_2}(\gamma)&\ceq a+(-1)^{e_1}c/4,
\\
t_3^{e_1,e_2}(\gamma)&\ceq d+(-1)^{e_2}c/4,
\\
4\cdot t_4^{e_1,e_2}(\gamma)&\ceq (-1)^{e_2}a+4b+(-1)^{e_1}(-1)^{e_2}c/4+(-1)^{e_1}d.
\end{align*}

\begin{lem}\label{disjoint}
For integers $e_1,e_2$, we define
\begin{align*}
\Gamma_1^{e_1,e_2}&\ceq\{\gamma\in\Gamma_0(4)\mid t_1^{e_1,e_2}(\gamma)=0\},
\\
\Gamma_2^{e_1,e_2}&\ceq\{\gamma\in\Gamma_0(4)\mid t_2^{e_1,e_2}(\gamma)=0\},
\\
\Gamma_3^{e_1,e_2}&\ceq\{\gamma\in\Gamma_0(4)\mid t_3^{e_1,e_2}(\gamma)=0,\ (-1)^{e_1}=(-1)^{e_2}\Rightarrow b\neq 0\},
\\
\Gamma_4^{e_1,e_2}&\ceq\{\gamma\in\Gamma_0(4)\setminus\{\pm I\}\mid t_4^{e_1,e_2}(\gamma)=0\},
\\
\Gamma_5^{e_1,e_2}&\ceq\left\{\gamma\in\Gamma_0(4)\midd \gamma\notin \Gamma_j^{^{e_1,e_2}},\ \forall j\in\{1,2,3,4\}\right\}.
\end{align*}
Then
$$\Gamma_0(4)=\coprod_{j=1}^{5}\Gamma_{j}^{e_1,e_2}\quad (\text{disjoint union}).$$
\end{lem}

\begin{proof}
Although it can be easily shown from 
$$\eta_{e_1}\Gamma_{0}(4)\eta_{e_2}=\left\{\begin{pmatrix}t_2^{e_1,e_2}(\gamma) & t_4^{e_1,e_2}(\gamma) \\ t_1^{e_1,e_2}(\gamma) & t_3^{e_1,e_2}(\gamma)\end{pmatrix}\midd \gamma\in\Gamma_{0}(4)\right\},$$
for the sake of the subsequent proof, we prove this lemma using a specific representation. Clearly,
\begin{equation}\label{gamma1}
\Gamma_1^{e_1,e_2}=\left\{\pm\begin{pmatrix}1&b\\0&1\end{pmatrix}\midd b\in\mathbb{Z}\right\}.
\end{equation}
If $a+(-1)^{e_1}c/4=0$, since $a$ and $c$ are coprime and $\det\gamma=1$, we have $a=-(-1)^{e_1}c/4=\pm 1$ and $(-1)^{e_1}4b+d=\pm 1$. Thus,
\begin{equation}\label{gamma2}
\Gamma_2^{e_1,e_2}=\left\{\pm\begin{pmatrix} 1&b\\ (-1)^{e_1+1}4&1+(-1)^{e_1+1}4b\end{pmatrix}\midd b\in\mathbb{Z}\right\}.
\end{equation}
Similarly, if $d+(-1)^{e_2}c/4=0$, we have $d=-(-1)^{e_2}c/4=\pm 1$ and $a+(-1)^{e_2}4b=\pm 1$. Thus,
\begin{equation}\label{gamma3}
\Gamma_3^{e_1,e_2}=\left\{\pm\begin{pmatrix}1+(-1)^{e_2+1}4b&b\\ (-1)^{e_2+1}4&1\end{pmatrix}\midd b\in\mathbb{Z},\ (-1)^{e_1}=(-1)^{e_2}\Rightarrow b\neq 0\right\}.
\end{equation}
If $(-1)^{e_2}a+4b+(-1)^{e_1}(-1)^{e_2}c/4+(-1)^{e_1}d=0$, by multiplying both sides by $a(\neq 0)$, we have
\begin{align}
(-1)^{e_1}+(a+(-1)^{e_1}c/4)(4b+(-1)^{e_2}a)=0,\label{gamma4-1}
\end{align}
Thus, we have $a+(-1)^{e_1}c/4=\pm1$ and $4b+(-1)^{e_2}a=\pm(-1)^{e_1+1}$. Hence, we have
\begin{align}
\gamma_{b}^{e_1,e_2}&\ceq \begin{pmatrix}(-1)^{e_1+e_2+1}+(-1)^{e_2+1}4b&b\\ 4((-1)^{e_1}+(-1)^{e_2}+(-1)^{e_1+e_2}4b)&(-1)^{e_1+e_2+1}+(-1)^{e_1+1}4b\end{pmatrix},\notag
\\
\Gamma_4^{e_1,e_2}&=\left\{\pm\gamma_{b}^{e_1,e_2}\midd \begin{array}{c}b\in\mathbb{Z}\\(-1)^{e_1+e_2+1}=1\Rightarrow b\neq 0\end{array}\right\}.\label{gamma4}
\end{align}
Therefore, we can show this lemma since there is no common intersection among $\Gamma_j^{e_1,e_2}\ (j=1,2,3,4)$.
\end{proof}

We define
\begin{equation*}
\begin{split}
S_{w,n,j}^{e_1,e_2}(\tau)
&\ceq\sum_{\gamma\in\Gamma_j^{e_1,e_2}}\left.\frac{1}{\tau^{\til{n}+1}}\right|_k\eta_{e_1}\gamma\eta_{e_2}
\\
&=\sum_{\gamma\in\Gamma_j^{e_1,e_2}}
\frac{1}{(t_1^{e_1,e_2}(\gamma)\tau+t_3^{e_1,e_2}(\gamma))^{n+1}(t_2^{e_1,e_2}(\gamma)\tau+t_4^{e_1,e_2}(\gamma))^{\til{n}+1}},
\end{split}
\end{equation*}
and set $S_{w,n,j}(\tau)\ceq S_{w,n,j}^{0,0}(\tau)+S_{w,n,j}^{1,1}(\tau)-S_{w,n,j}^{0,1}(\tau)-S_{w,n,j}^{1,0}(\tau)$. Then, we have
\begin{equation}\label{decom-R}
R_{w,n,\chi_4}(\tau)=\frac{\alpha_{w,n}^{-1}}{2i}(S_{w,n,1}+S_{w,n,2}+S_{w,n,3}+S_{w,n,4}+S_{w,n,5})
\end{equation}
from \cref{eq:f-chi4}, \cref{lem:RR-explicit,disjoint}.

\begin{prop}\label{prop:RR-explicit}
For any positive integers $m$ and $n$ such that $w>m>n$, $m>\til{n}$, and $n+m$ is odd, we have
\begin{align*}
r_m(R_{w,n,\chi_4})
=-\frac{1}{w!}\Biggl((-1)^{n}&n!m!\left(1-\frac{1}{2^{m+1-\til{n}}}\right)\frac{B_{m+1-\til{n}}}{(m+1-\til{n})!}
\\
&+4^{\til{n}-n}\til{n}!m!\left(1-\frac{1}{2^{m+1-n}}\right)\frac{B_{m+1-n}}{(m+1-n)!}\Biggr).
\end{align*}
\end{prop}

\begin{proof}
First, by \cref{gamma1},
$$S_{w,n,1}(\tau)
=4^{\til{n}+2}\sum_{b\in\mathbb{Z}}\frac{\chi_4(b-1)}{(4\tau+b)^{\til{n}+1}}
=-\frac{4^{\til{n}+2}}{2^{\til{n}+1}}\psi_{\til{n}+1}(\tau).$$
Thus, if $m>\til{n}$, then
\begin{align*}
\int_0^{i\infty}S_{w,n,1}(\tau)\tau^{m}d\tau
&=-\frac{4^{\til{n}+2}}{2^{\til{n}+1}}\int_0^{i\infty}\psi_{\til{n}+1}(\tau)\tau^{m}d\tau
\\
&=-\frac{4^{\til{n}+2}}{2^{\til{n}+1}}\frac{(-2\pi i)^{\til{n}+1}}{\til{n}!}\frac{1}{2^{\til{n}}}\int_0^{i\infty}\sum_{b=1}^{\infty}\chi_0(b)b^{\til{n}}q^b\tau^{m}d\tau
\\
&=-2^3\frac{(-2\pi i)^{\til{n}+1}}{\til{n}!}\sum_{b=1}^{\infty}\chi_0(b)b^{\til{n}}\int_0^{i\infty}e^{2\pi ib\tau}\tau^{m}d\tau
\\
&=-2^3(-2\pi i)^{\til{n}-m}\frac{m!}{\til{n}!}L(\chi_0,m+1-\til{n})
\\
&=(-8\pi i)\frac{m!}{\til{n}!}\left(1-\frac{1}{2^{m+1-\til{n }}}\right)\frac{B_{m+1-\til{n}}}{(m+1-\til{n})!}
\end{align*}
Second, by \cref{gamma2},
$$S_{w,n,2}(\tau)
=(-1)^{n}4^{\til{n}+2}\sum_{b\in\mathbb{Z}}\frac{\chi_4(b-1)}{(4\tau+b)^{n+1}}
=(-1)^{n+1}\frac{4^{\til{n}+2}}{2^{n+1}}\psi_{n+1}(\tau).$$
Thus, if $m>n$, then
\begin{align*}
\int_0^{i\infty}S_{w,n,2}(\tau)\tau^{m}d\tau
&=(-1)^{n+1}2^{3}4^{\til{n}-n}(-2\pi i)^{n-m}\frac{m!}{n!}L(\chi_0,m+1-n)
\\
&=(-1)^{n}4^{\til{n}-n}(-8\pi i)\frac{m!}{n!}\left(1-\frac{1}{2^{m+1-n}}\right)\frac{B_{m+1-n}}{(m+1-n)!}.
\end{align*}
In the following part of this proof, we can justify the interchange of integration and infinite summation in the same way as in the argument presented in the proof of \cite[Proposition 2.3]{Fuk-Yan}. Third, by \cref{gamma3},
\begin{align*}
S_{w,n,3}(\tau)
&=\frac{4}{(4\tau)^{n+1}}\sum_{b\in\mathbb{Z}\setminus\{0\}}\frac{\chi_4(b+1)}{(b\tau-1/4)^{\til{n}+1}}.
\end{align*}
Therefore, if $m>n+1$, we have
\begin{align*}
\int_0^{i\infty}S_{w,n,3}(\tau)\tau^md\tau
&=\frac{4}{4^{n+1}}\sum_{b\in\mathbb{Z}\setminus\{0\}}\int_{0}^{i\infty}\frac{\chi_4(b+1)}{(b\tau-1/4)^{\til{n}+1}}\tau^{m-n-1}d\tau
\\
&=\frac{2}{4^{n+1}}\sum_{b\in\mathbb{Z}\setminus\{0\}}\int_{-i\infty}^{i\infty}\frac{\chi_4(b+1)}{(b\tau-1/4)^{\til{n}+1}}\tau^{m-n-1}d\tau.
\end{align*}
To compute each integral in the sum, we employ the standard method using the Cauchy integral theorem, i.e., first considering the integral along the half-circle of radius $R>0$ centered at the origin (along the imaginary axis and right or left arc according to the sign of $b$, so as to escape the unique pole $1/(4b)$ inside the path of integration), then taking the limit $R\to \infty$. We then obtain that the value of each integral is zero, resulting in the conclusion
$$\int_0^{i\infty}S_{w,n,3}(\tau)\tau^md\tau=0.$$
If $m=n+1$, we have
\begin{align*}
\int_0^{i\infty}&S_{w,n,3}(\tau)\tau^md\tau
\\
&=\frac{1}{4^n}\lim_{\ep\to\infty}\sum_{b\in\mathbb{Z}\setminus\{0\}}\int_{i\ep}^{i\infty}\frac{\chi_4(b+1)}{(b\tau-1/4)^{\til{n}+1}}d\tau
\\
&=\frac{1}{4^n\til{n}}\lim_{\ep\to\infty}\sum_{b\in\mathbb{Z}\setminus\{0\}}\frac{\chi_4(b+1)}{b(bi\ep-1/4)^{\til{n}}}
\\
&=\frac{1}{4^n\til{n}}\lim_{\ep\to\infty}\sum_{b\in\mathbb{Z}\setminus\{0\}}\left(\frac{2}{4b\ep(4bi\ep-1/4)^{\til{n}}}-\frac{1}{2b\ep(2bi\ep-1/4)^{\til{n}}}\right)\ep
\\
&=\frac{1}{4^n\til{n}}\int_{-\infty}^{\infty}\Biggl(\frac{2}{4x(4ix-1/4)^{\til{n}}}
-\frac{1}{2x(2ix-1/4)^{\til{n}}}\Biggr)dx.
\end{align*}
And by using residue theorem at the points $x=1/16i,1/8i$, we obtain
\begin{align*}
\int_0^{i\infty}S_{w,n,3}(\tau)\tau^md\tau
&=\frac{2}{4^{n+1}\til{n}}\left((-1)^{\til{n}-1}\frac{(16i)^{\til{n}}}{(4i)^{\til{n}}}-(-1)^{\til{n}-1}\frac{(8i)^{\til{n}}}{(2i)^{\til{n}}}\right)
=0.
\end{align*}
Fourth, by \cref{gamma4},
\begin{align*}
S_{w,n,4}(\tau)
=\frac{4}{4^{n+1}}\frac{1}{\tau^{\til{n}+1}}\sum_{b\in\mathbb{Z}\setminus\{0\}}
\frac{\chi_4(b-1)}{(b\tau+1/4)^{n+1}}.
\end{align*}
Then, similar to $S_{w,n,3}(\tau)$, if $m\geq \til{n}+1$, we obtain
$$
\int_0^{i\infty}S_{w,n,4}(\tau)\tau^md\tau=0
$$
Finally, we consider $S_{w,n,5}(\tau)$. We can easily check that $\begin{pmatrix}a&-b\\-c&d\end{pmatrix}\in \Gamma_{5}^{e_1+1,e_2+1}$ if $\begin{pmatrix}a&b\\c&d\end{pmatrix}\in \Gamma_{5}^{e_1,e_2}$, and this implies
\begin{align*}
\int_{0}^{i\infty}(S_{k,n,5}^{e_1,e_2}(\tau)+S_{k,n,5}^{e_1+1,e_2+1}(\tau))\tau^{m}d\tau
&=\sum_{\gamma\in\Gamma_{5}^{e_1,e_2}}\int_{-i\infty}^{i\infty}\frac{\tau^{m}d\tau}{(t_1^{e_1,e_2}(\gamma)\tau+t_3^{e_1,e_2}(\gamma))^{n+1}(t_2^{e_1,e_2}(\gamma)\tau+t_4^{e_1,e_2}(\gamma))^{\til{n}+1}}
\end{align*}
if $m<w$. On the other hand, the sign of two non-zero poles
$$-\frac{t_4^{e_1,e_2}(\gamma)}{t_2^{e_1,e_2}(\gamma)}\ ,\quad -\frac{t_3^{e_1,e_2}(\gamma)}{t_1^{e_1,e_2}(\gamma)}$$
matches because $t_1^{e_1,e_2}(\gamma)t_4^{e_1,e_2}(\gamma),t_2^{e_1,e_2}(\gamma)t_3^{e_1,e_2}(\gamma)\in \mathbb{Z}\setminus\{0\}$ and 
$$\det(\eta_{e_1}\gamma\eta_{e_2})=t_2^{e_1,e_2}(\gamma)t_3^{e_1,e_2}(\gamma)-t_1^{e_1,e_2}(\gamma)t_4^{e_1,e_2}(\gamma)=1.$$
Thus, similar to $S_{w,n,3}(\tau)$, if $m<w$, we obtain
$$\int_{0}^{i\infty}S_{k,n,5}(\tau)\tau^{m}d\tau=0.$$
Therefore, by \cref{decom-R}, we obtain this proposition.
\end{proof}

\begin{cor}\label{cor:matrix-A} Let
\begin{equation*}
c_{m,n}\ceq 4^{n-m}n!\til{m}!\left(1-\frac{1}{2^{n+1-m}}\right)\frac{B_{n+1-m}}{(n+1-m)!}
\end{equation*}
and $a_{m,n}\ceq c_{m,n}+(-1)^n c_{m,\til{n}}$. Then, we have
\begin{equation}
A_w=\left(\frac{1}{w!}4^{w-(2i+2j-1)}(a_{2j,2i-1}\delta_{j<i}-a_{2i-1,2j}\delta_{j\geq i})\right)_{1\leq i,j\leq \til{d}}\ .
\end{equation}
Here, $\delta_*$ returns $1$ when the condition $*$ is satisfied and returns $0$ otherwise.
\end{cor}

\begin{proof}
Let
\begin{equation*}
c'_{m,n}\ceq 4^{m-n}m!\til{n}!\left(1-\frac{1}{2^{m+1-n}}\right)\frac{B_{m+1-n}}{(m+1-n)!}
\end{equation*}
and $a'_{m,n}\ceq c'_{m,n}+(-1)^n c'_{m,\til{n}}$. For positive integers $m$ and $n$ such that $m>\til{n}$ and $m+n$ is odd, we have
\begin{equation*}
r_{m}(R_{w,n,\chi_4}) = -\frac{4^{\til{n}-m}}{w!}\cdot \left\{ \begin{array}{cc}
a'_{m,n} & (m>n) \\
-a'_{n,m} & (m<n)
\end{array}\right.
\end{equation*}
by \cref{prop:RR-explicit} and $r_m(R_{w,n,\chi_4})=-r_n(R_{w,m,\chi_4})$. Moreover, by \cref{lem:atkin-lehner} and $a'_{\til{m},\til{n}}=a_{m,n}$, for positive integers $m$ and $n$ such that $m<\til{n}$ and $m+n$ is odd, we have
\begin{equation*}
r_{m}(R_{w,n,\chi_4}) = -4^{\til{n}+\til{m}-n-m}r_{\til{m}}(R_{w,\til{n},\chi_4}) = \frac{4^{\til{n}-m}}{w!}\cdot \left\{ \begin{array}{cc}
a_{m,n} & (m<n) \\
-a_{n,m} & (m>n)
\end{array}\right. .
\end{equation*}
Hence, we obtain this corollary from $2j<\til{2i-1}$ for $1\leq i,j\leq \til{d}=[(k-2)/4]$. 
\end{proof}

Recall that our purpose is to show $\det A_w\neq 0$. By \cref{cor:matrix-A}, it is enough to prove that the determinant of the matrix
\begin{equation}
\til{A}_w\ceq\Bigl(a_{2j,2i-1}\delta_{j<i}-a_{2i-1,2j}\delta_{j\geq i}\Bigr)_{1\leq i,j\leq \til{d}}
\end{equation}
is non-zero.

\begin{prop}
For any even integer $w\geq 6$, $\det \til{A}_w\neq 0$.
\end{prop}

\begin{proof}
Let $b_{m,n}:=(-1)^{n}c_{m,\til{n}}$, then $a_{m,n}=b_{m,n}+c_{m,n}$. We consider the 2-adic valuation $\ord_2$ of $b_{m,n}$ and $c_{m,n}$. By von Staudt--Clausen theorem, we have $\ord_2(B_n)=-1\ (n\geq 2:\text{even})$. Thus, we have
\begin{equation*}
\ord_2(b_{m,n})=\ord_2(b_{n,m})=\til{n}-m-2+\ord_2\left(\frac{\til{n}!\til{m}!}{(\til{n}+1-m)!}\right),\quad (\til{n}>m).
\end{equation*}
Furthermore, by using $\ord_2(n!)<n$, we have
\begin{equation*}
\ord_2(b_{m,n})>\ord_2(\til{n}!\til{m}!)-3.
\end{equation*}
If $n\neq \til{n}$, i.e., $n<\til{n}$, then
\begin{equation}\label{adic}
\ord_2(b_{m,n})>\ord_2(n!\til{m}!)-2,
\end{equation}
otherwise, i.e., $n=\til{n}$, then 
\begin{equation}\label{adic2}
a_{m,n}=2c_{m,n}\quad \text{and}\quad \ord_2(a_{m,n})=\ord_2(c_{m,n})+1.
\end{equation}
Note that $a_{m,n}$ with $n=\til{n}$ as an element of $\til{A}_w$ occurs only when $w\equiv 0\modd 4$ and $n=2\til{d}$.
\\
\\
\underline{Case 1}. We consider the case $i=j$, i.e., consider $a_{2j-1,2j}$. Then, we have
\begin{equation*}
c_{2j-1,2j}=\frac{1}{4}(2j)!(w-(2j-1))! \quad \text{and}\quad
\ord_2(c_{2j-1,2j})=\ord_2((2j)!(w-2j)!)-2.
\end{equation*}
Thus, by \cref{adic,adic2}, we have
\begin{equation*}
\ord_2(a_{2j-1,2j})
=\left\{
\begin{array}{cc}
\ord_2((2j)!(w-2j)!)-2 & (j\neq\til{d})\\
\ord_2((2j)!(w-2j)!)-1 & (j=\til{d})
\end{array}\right. .
\end{equation*}
\underline{Case 2}. We consider the case $i=j+1$, i.e., consider $a_{2j,2j+1}$. Then,
\begin{equation*}
c_{2j,2j+1}=\frac{1}{4}(2j+1)!(w-2j)! \quad \text{and}\quad
\ord_2(c_{2j,2j+1})=\ord_2((2j)!(w-2j)!)-2.
\end{equation*}
Thus, by \cref{adic}, we have
\begin{equation*}
\ord_2(a_{2j,2j+1})=\ord_2((2j)!(w-2j)!)-2.
\end{equation*}
\underline{Case 3}. We consider the case $j=i+h\ (h\geq 1)$, i.e., consider $a_{2(j-h)-1,2j}$. Then,
\begin{align*}
c_{2j-2h-1,2j}
&=4^{2h+1}(2j)!(w-(2j-2h-1))!\left(1-\frac{1}{2^{2h+2}}\right)\frac{B_{2h+2}}{(2h+2)!}.
\end{align*}
Thus, by \cref{adic} and 
\begin{align*}
\ord_2(c_{2j-2h-1,2j})
&=2h-1+\ord_2\left(\frac{(2j)!(w-(2j-2h-1))!}{(2h+2)!}\right)
\\
&\geq \ord_2((2j)!(w-(2j-2h-1))!)-2>\ord_2((2j)!(w-2j)!)-2,
\end{align*}
we have
\begin{equation*}
\ord_2(a_{2j-2h-1,2j})>\left\{
\begin{array}{cc}
\ord_2((2j)!(w-2j)!)-2 & (j\neq\til{d})\\
\ord_2((2j)!(w-2j)!)-1 & (j=\til{d})
\end{array}\right. .
\end{equation*}
\underline{Case 4}. We consider the case $i=j+h\ (h\geq 2)$, i.e., consider $a_{2j,2(j+h)-1}$. Then,
\begin{align*}
c_{2j,2j+2h-1}
&=4^{2h-1}(2j+2h-1)!(w-2j)!\left(1-\frac{1}{2^{2h}}\right)\frac{B_{2h}}{(2h)!}.
\end{align*}
Thus, by \cref{adic} and
\begin{align*}
\ord_2(c_{2j,2j+2h-1})
\geq \ord_2((2j+2h-1)!(w-2j)!)-2 >\ord_2((2j)!(w-2j)!)-2,
\end{align*}
we have
$$\ord_2(a_{2j,2j+2h-1}) >\ord_2((2j)!(w-2j)!)-2.$$
Therefore, by the definition of the determinant, we obtain
$$\ord_2(\det \til{A}_w)=\sum_{j=1}^{[w/4]}\ord_2((2j)!(w-2j)!)-2\left[\frac{w}{4}\right]+\delta_{w\equiv 0(4)},$$
especially, $\ord_2(\det \til{A}_w)<\infty$. Consequently, we have $\det \til{A}_w\neq 0$.
\end{proof}

This completes the proof of \cref{thm:dimension}.

The following corollary is an immediate consequence of what we have shown.

\begin{cor}\label{cor:R-indep}
Let $S^{\text{new}}_k$ is the $\mathbb{C}$-vector space of generated by $\til{B}_k$. For any even integer $k\geq 6$, we have the following three statements.
\begin{enumerate}
\item Each of the $\{ R_{w,2i}\mid 1\leq i\leq \til{d}\}$ and $\{R_{w,2i-1}\mid 1\leq i\leq \til{d}\}$ is a basis of $S^{\text{new}}_k$.
\item For $f\in S^{\text{new}}_k$, if $r_1(f_{\chi_4})=r_3(f_{\chi_4})=\cdots =r_{2\til{d}-1}(f_{\chi_4})=0$, then $f=0$.
\item For $f\in S^{\text{new}}_k$, if $r_2(f_{\chi_4})=r_4(f_{\chi_4})=\cdots =r_{2\til{d}}(f_{\chi_4})=0$, then $f=0$.
\end{enumerate}
\end{cor}

We define the period polynomial of $f$ by
$$r(f)(X,Y)\ceq \int_{0}^{i\infty}f(\tau)(X-\tau Y)^{w}d\tau .$$

We use the symbols defined in \cref{sec:ImDH}.

\begin{prop}
For any even integer $k\geq 6$, we have the isomorphism
$$S_{k}^{\text{new}}\ni f\longmapsto r(f_{\chi_4})^{\text{od}}(X,Y)\in \Ker \til{\Delta} \otimes \mathbb{C}.$$
Here, $\til{\Delta}=(1+\delta)+(1-\delta)A(1+\ep)$ and $A=\begin{pmatrix}1&0\\0&4\end{pmatrix}$.
\end{prop}

\begin{proof}
By \cref{lem:atkin-lehner}, we have $4^{n}r_n(f_{\chi_4})=(-1)^{n}4^{\til{n}}r_{\til{n}}(f_{\chi_4})$. Furthermore, we have
\begin{align*}
r(f_{\chi_4})^{\text{od}}(X,4Y)
&=\sum_{n=0:\text{odd}}^{k-2}(-1)^{n}\binom{w}{n}4^{n}r_{n}(f_{\chi_4})Y^{n}X^{\til{n}}d\tau
\\
&=\sum_{n=0:\text{odd}}^{w}\binom{w}{n}4^{\til{n}} r_{\til{n}}(f_{\chi_4}) Y^{n}X^{\til{n}}
=-r(f_{\chi_4})^{\text{od}}(Y,4X).
\end{align*}
Hence, $r(f_{\chi_4})^{\text{od}}(X,Y)\in \Ker\til{\Delta}\otimes \mathbb{C}$. Moreover, by $\Img \til{\Delta} = V_k^{\text{ev}}\oplus V_k^{\text{od}}|_{A(1+\ep)}$ and 
$$V_k^{\text{od}}|_{A(1+\ep)}=V_k^{\text{od}}|_{(1+\ep)}=\left\{\sum_{n=0:\text{odd}}^{(k-2)/2}a_{n}(X^{n}Y^{w-n}+X^{w-n}Y^{n})\middle| a_n\in\mathbb{Q}\right\},$$
we have $\dim\Img \til{\Delta} = k/2+[k/4]=k-1-[(k-2)/4]$ and thus $\dim \Ker\til{\Delta}=[(k-2)/4]$. Therefore, the map is injective between spaces of equal dimensions by \cref{cor:R-indep} and thus this is an isomorphism.
\end{proof}

\section{The relations among double $\til{T}$-values}\label{sec:relation}

The rational number $\lambda_{k,r}$ such that $\til{H}_{r}\til{H}_{k-r}-\lambda_{k,r}\til{G}_{k}\in S\DHk$ given by
$$\lambda_{k,r}=\frac{4L(\chi_4,r)L(\chi_4,k-r)}{L(\chi_{0},k)}=\frac{1}{2^{k-1}(2^{k}-1)}\frac{k!}{(r-1)!(k-r-1)!}\frac{E_{r-1}E_{k-r-1}}{B_k}$$
by \cref{eq:zeta-Bernoulli,eq:L-Euler}. Thus, by \cref{thm:shuffle,thm:dimension}, we obtain the independent relations among double $\til{T}$-values: for an integer $1\leq j\leq [(k-2)/4]$ and $r=2j+1$, let
$$\til{a}_{k,j,p}\ceq \binom{k-p-1}{r-1}+\binom{k-p-1}{k-r-1}-\frac{(1+\delta_{1,p})}{2^{p+1}}\frac{k}{2^{k}-1}\binom{k-2}{r-1}\frac{E_{r-1}E_{k-r-1}}{B_k},$$
then
\begin{equation}\label{eq:relation-from-modular}
\sum_{p=1}^{k-1}\til{a}_{k,j,p}\til{T}(p,k-p)=0.
\end{equation}
If \Cref{conj:K-Tsu} below is correct (i.e., if the evaluation in \cref{cor:Ttil-dim} is the best possible), then the relations among the double $\til{T}$-values should be generated by these.

\begin{conj}[{\cite[Conjecture 2.12, 1)]{K-Tsu}}]\label{conj:K-Tsu}
For $N=2$ and $4$, even integer $k\geq 4$, and $1\leq j\leq (k-2)/2$, define the polynomial $S_{N,k,j}(X)$ with rational coefficients by
\begin{align*}
\til{S}_{N,k,j}(X)\ceq\frac{N^{k-2j-1}}{k-2j}&X^{k-2}B^{0}_{k-2j}\left(\frac{1}{NX}\right)-\frac{1}{2j}B^{0}_{2j}(X)
\\
&-\frac{kB_{2j}B_{k-2j}}{2j(k-2j)B_k}\left(\frac{1-2^{-2j}}{1-2^{-k}}\frac{X^{k-2}}{N}-\frac{1-2^{-k+2j}}{1-2^{-k}}\frac{1}{N^{2j}}\right),
\end{align*}
where
$$B_{n}^{0}(X)\ceq \sum_{\substack{0\leq j\leq n\\j:\text{even}}}\binom{n}{j}B_{j}X^{n-j}.$$
We further define
\begin{align*}
P_{N,k,j}(X,Y)&\ceq (-2X+2Y)^{k-2}\til{S}_{N,k,j}\left(\frac{X+Y}{-2X+2Y}\right)
\end{align*}
and write the polynomial $P^{\text{ev}}_{N,k,j}(X+Y,Y)$ as
$$P^{\text{ev}}_{N,k,j}(X+Y,Y)=\sum_{i=1}^{k-1}a_{N,k,j,i}\binom{k-2}{i-1}X^{i-1}Y^{k-i-1}.$$
Then we have the following relation among the double $\til{T}$-values:
\begin{equation}\label{eq:relation-from-period}
\sum_{i=1}^{k-1}a_{N,k,j,i}\til{T}(i,k-i)=0.
\end{equation}
The $\mathbb{Q}$-vector space $V_{4,k}$ spanned by $P_{4,k,j}^{\text{ev}}(X,Y)\ (1\leq j\leq (k-2)/2)$ is of dimension $[(k-2)/4],$ which we conjecture to be equal to the number of independent relations among double $\til{T}$-values of weight $k$. The polynomials $P^{\text{ev}}_{2,k,j}(X,Y)$ are contained in $V_{4,k}$, and span the subspace of dimension $[(k-2)/6]$.
\end{conj}

The polynomial $\til{S}_{N,k,j}(X)$ is the period polynomial $r^{\text{ev}}(R_{\Gamma_{0}(N),w,2j-1})(X)$ in the work of Fukuhara and Yang \cite{Fuk-Yan}. $R_{\Gamma_{0}(N),w,2j-1}$ is an analogue of $R_{w,n}$ introduced in this paper and is a cusp form of level $N$.

We can confirm that the relations \cref{eq:relation-from-period} are generated by the relations \cref{eq:relation-from-modular} at a low weight through numerical computation. However, we do not have a conjecture about the general form of the coefficients.

\begin{example}
\begin{enumerate}
\item For $N=4$, $k=6$, and $j=1$, we have
\begin{align*}
\sum_{i=1}^{k-1}a_{4,6,1,i}X^{i-1}&=-8-4X-\frac{2}{3}X^2+X^3+X^4,
\\
\sum_{i=1}^{k-1}\til{a}_{6,1,i}X^{i-1}&=6+3X+\frac{1}{2}X^2-\frac{3}{4}X^3+\frac{3}{4}X^4.
\end{align*}
Therefore, we have
\begin{equation*}
\sum_{i=1}^{k-1}a_{4,6,1,i}X^{i-1}=-\frac{4}{3}\sum_{i=1}^{k-1}\til{a}_{6,1,i}X^{i-1}.
\end{equation*}

\item For $N=2$, $k=8$, and $j=1$, we have
\begin{align*}
\sum_{i=1}^{k-1}a_{2,8,1,i}X^{i-1}&=-\frac{1792}{51}-\frac{896}{51}X-\frac{5632}{765}X^2-\frac{192}{85}X^3+\frac{224}{765}X^4+\frac{80}{51}X^5+\frac{80}{51}X^6,
\\
\sum_{i=1}^{k-1}\til{a}_{8,1,i}X^{i-1}&=\frac{210}{17}+\frac{105}{17}X+\frac{44}{17}X^2+\frac{27}{34}X^3-\frac{7}{68}X^4-\frac{75}{136}X^5-\frac{75}{136}X^6.
\end{align*}
Therefore, we have
\begin{equation*}
\sum_{i=1}^{k-1}a_{2,8,1,i}X^{i-1}=-\frac{128}{45}\sum_{i=1}^{k-1}\til{a}_{8,1,i}X^{i-1}.
\end{equation*}

\item For $N=4$, $k=10$, and $j=3$, we have
\begin{align*}
\sum_{i=1}^{k-1}a_{4,10,3,i}X^{i-1}&=-\frac{6144}{31}-\frac{3072}{31}X-\frac{25808}{651}X^2-\frac{2152}{217}X^3\\
&\quad +\frac{3824}{3255}X^4+\frac{640}{217}X^5+\frac{1270}{651}X^6+\frac{45}{31}X^7+\frac{45}{31}X^8,
\end{align*}
and for $k=4$, we have
\begin{align*}
\sum_{i=1}^{k-1}\til{a}_{10,1,i}X^{i-1}&=\frac{28}{31}+\frac{14}{31}X+\frac{69}{31}X^2+\frac{193}{62}X^3\\
&\quad +\frac{317}{124}X^4+\frac{317}{248}X^5+\frac{69}{496}X^6-\frac{427}{992}X^7-\frac{427}{992}X^8,
\\
\sum_{i=1}^{k-1}\til{a}_{10,2,i}X^{i-1}&=\frac{2580}{31}+\frac{1295}{31}X+\frac{985}{62}X^2+\frac{365}{124}X^3\\
&\quad -\frac{379}{248}X^4-\frac{875}{496}X^5-\frac{875}{992}X^6-\frac{875}{1984}X^7-\frac{875}{1984}X^8.
\end{align*}
Therefore, we have
\begin{equation*}
\sum_{i=1}^{k-1}a_{4,10,3,i}X^{i-1}=-\frac{20}{21}\sum_{i=1}^{k-1}\til{a}_{10,1,i}X^{i-1}-\frac{248}{105}\sum_{i=1}^{k-1}\til{a}_{10,2,i}X^{i-1}.
\end{equation*}
\end{enumerate}
\end{example}

\end{document}